\documentclass{amsart}
\usepackage{graphicx} 
\usepackage{amsfonts, amsmath, amssymb, amsthm}
\usepackage{ stmaryrd } 
\usepackage{xypic}
\usepackage[margin=1in]{geometry}
\usepackage[dvipsnames]{xcolor} 
\usepackage[hidelinks]{hyperref}
\hypersetup{linktocpage,bookmarksopen = true, bookmarksopenlevel = 1, colorlinks=true,
linkcolor = Bittersweet, citecolor = RoyalBlue , urlcolor = gray, pdfstartview = FitR}

\newtheorem{theorem}{Theorem}
\newtheorem{informal-theorem}[theorem]{Informal Theorem}
\newtheorem{lemma}[theorem]{Lemma}
\newtheorem{proposition}[theorem]{Proposition}

\theoremstyle{definition}
\newtheorem{definition}[theorem]{Definition}
\theoremstyle{remark}
\newtheorem{remark}[theorem]{Remark}
\newtheorem{algorithm}[theorem]{Algorithm}

\numberwithin{theorem}{section}

\definecolor{myorange}{rgb}{0.9, 0.55, 0.3}
\definecolor{mygreen}{rgb}{0.35, 0.71, 0.0}
\definecolor{mybrown}{rgb}{0.63, 0.32, 0.18}

\newcommand{\CC}{\mathbb{C}}

\newcommand{\QQ}{\mathbb{Q}}
\newcommand{\ZZ}{\mathbb{Z}}
\newcommand{\GG}{\mathbb{G}}

\newcommand{\FF}{\mathbb{F}}

\newcommand{\End}{\operatorname{End}}

\newcommand{\Pic}{\operatorname{Pic}}
\newcommand{\Div}{\operatorname{Div}}

\renewcommand{\div}{\operatorname{div}}
\newcommand{\Tr}{\operatorname{Tr}}

\title{Sesquilinear pairings on elliptic curves}
\author{Katherine E. Stange}
\date{\today}
\thanks{This work has been supported by NSF-CAREER CNS-1652238, NSF DMS-2401580, and an AMS Joan and Joseph Birman Fellowship 2025-26.}
\subjclass[2020]{Primary: 
11G05, 
14H52} 

\keywords{Elliptic curves, Weil pairing, Tate-Lichtenbaum pairing, complex multiplication}

\begin{document}

\maketitle

\begin{abstract}
	Let $E$ be an elliptic curve with complex multiplication by a ring $R$, where $R$ is an order in an imaginary quadratic field or quaternion algebra.  We define sesquilinear pairings ($R$-linear in one variable and $R$-conjugate linear in the other), taking values in an $R$-module, generalizing the Weil and Tate-Lichtenbaum pairings.
\end{abstract}

\section{Introduction}

The Weil and Tate-Lichtenbaum pairings are bilinear pairings on an elliptic curve $E$ with values in the multiplicative group $\GG_m$.  In the situation of complex multiplication, the points of the elliptic curve form more than just a $\ZZ$-module, but also an $R$-module, for some ring $R$ which is an order in either an imaginary quadratic field or a quaternion algebra, both of which come equipped with an involution which we call \emph{conjugation}.  It is natural then to hope for a pairing with some type of $R$-linearity.
In this paper, we generalize these classical pairings to take values in an $R$-module, so that the pairings can become \emph{sesquilinear}, or \emph{conjugate linear} in the following sense.  
If $R$ is commutative, an \emph{$R$-sesquilinear pairing} (conjugate linear on the left) is a bilinear pairing
$\langle \cdot, \cdot \rangle$ on a pair of $R$-modules, taking values in another $R$-module, that satisfies
\[
\langle \gamma x , \delta y \rangle = {\delta}{\overline\gamma}  \langle x,y \rangle, \text{ for all } \gamma, \delta \in R.
\]
Alternatively, if $\langle \gamma x , \delta y \rangle = {\overline\delta}{\gamma}  \langle x,y \rangle$, we say it is conjugate linear on the right.
In the case that $R$ is non-commutative, we also consider a twisted version; see Section~\ref{sec:sesq}.  For the remainder of the introduction, we assume $R$ is commutative; small adjustments are needed in the non-commutative case.

The Weil and Tate-Lichtenbaum pairings can be defined on divisor classes in $\Pic^0(E)$.  By considering instead $\Pic_R^0(E) := R \otimes_\ZZ \Pic^0(E)$, we have an $R$-module structure on divisor classes.  To accommodate the values of the pairing, considering $\GG_m$ as a $\ZZ$-module in multiplicative notation, we can extend scalars to $R$, writing $\GG_m^{\otimes_\ZZ R}$.  (This multiplicative tensor notation is not without its pitfalls; see the end of the introduction for further discussion.)  Write $M[\alpha]$ for the $\alpha$-torsion in an $R$-module $M$.
For each $\alpha \in R$, we obtain Galois invariant sesquilinear pairings (conjugate linear on the right),
\begin{align*}
	{W}_\alpha &: \Pic_R^0(E)[\overline{\alpha}] \times \Pic_R^0(E)[\alpha] \rightarrow \GG_m^{\otimes_\ZZ R}[\overline{\alpha}], \\
	{T}_\alpha &: \Pic_R^0(E)[\overline{\alpha}] \times \Pic_R^0(E) / [\alpha]\Pic_R^0(E) \rightarrow \GG_m^{\otimes_\ZZ R} / (\GG_m^{\otimes_\ZZ R})^{\overline{\alpha}},
\end{align*}
generalizing the classical Weil and Tate-Lichtenbaum pairings (these do not restrict to the classical pairings, but restrict to a sesquilinearization of such; see Proposition~\ref{prop:tate-n} and the discussion afterward).  The pairing $W_\alpha$ is also conjugate skew-Hermitian in the sense that
    \[
    W_\alpha(D_P,D_Q) = \overline{W_{\overline{\alpha}}(D_Q,D_P)}^{-1}.
\]
These are defined by essentially imitating the definition of the classical pairings, including extending Weil reciprocity to $R$-divisors (see Theorem~\ref{thm: weilrecipgen}).

However, this formal exercise is most interesting when applied to a curve with endomorphism ring containing a copy of $R$.  Consider an exact sequence
\begin{equation*}
\xymatrix{
0 \ar[r] &E \ar[r]^-{\eta} &\Pic^0_{R}(E) \ar[r]^-\epsilon &E \ar[r] &0
}
\end{equation*}
given by
\[
  \epsilon: \sum_i \alpha_i (P_i) \mapsto \sum_i [\alpha_i]P_i, \quad \eta: P \mapsto ([-\tau]P) - (\mathcal{O}) + \tau \cdot \left( (P) - (\mathcal{O}) \right),
\]
where $[\alpha]P$ is the image of $P$ under multiplication-by-$\alpha$, and $R = \ZZ + \tau \ZZ$ (Section~\ref{sec:CM}).  The map $\eta$ twists the $R$-action in the sense that $\eta([\alpha]P) = \overline\alpha \eta(P)$.  
By restricting the pairing to the left-hand $E$ in the exact sequence, we obtain Galois invariant pairings
\begin{align*}
	\widehat{W}_\alpha &: E[\overline{\alpha}] \times E[\alpha] \rightarrow \GG_m^{\otimes_\ZZ R}[{\alpha}], \\
	\widehat{T}_\alpha &: E[\overline{\alpha}] \times E / [\alpha]E \rightarrow \GG_m^{\otimes_\ZZ R} / (\GG_m^{\otimes_\ZZ R})^{{\alpha}},
\end{align*}
which are $R$-sesquilinear (now conjugate linear on the left because of the twisting of $\eta$) in the sense that for all $\gamma, \delta \in R$ and $P \in E[\overline\alpha]$, $Q \in E$,
\[
\widehat{T}_\alpha( [\gamma]P, [\delta]Q) = \widehat{T}_\alpha(P,Q)^{\delta \overline\gamma},
\]
and similarly for $\widehat{W}_\alpha$.  These pairings are now defined on points of $E$, respecting the endomorphism action of $R$, which is the author's main goal.  Whereas the pairings $T_\alpha$ and $W_\alpha$ are `formal' in the sense that we simply extend scalars in domain and codomain, the pairings $\widehat{T}_\alpha$ and $\widehat{W}_\alpha$ obtained by pulling back are now interacting directly with the endomorphism ring of a CM curve.

When $R$ is non-commutative, a similar construction is possible, but sesquilinearity in one entry is twisted by an action of $\overline\alpha$ (Section~\ref{sec:sesq}).

In the case that $\alpha = n \in \ZZ$, these pairings can be interpreted as a `sesquilinearization' of the usual Weil and Tate-Lichtenbaum pairings.  For example if
\[
t_n : E[n] \times E/[n]E \rightarrow \GG_m/\GG_m^n
\]
represents the usual Tate-Lichtenbaum pairing, and $R = \ZZ + \tau \ZZ$, then (Theorem~\ref{thm:cmtatereduc})
    \[
     \widehat{T}_n (P,Q) = 
   \left( 
	   t_n(P,Q)^{2N(\tau)}
	   t_n([-\tau]P,Q)^{\Tr(\tau)} \right) \left(
	   t_n([\tau - {\overline\tau}]P,Q)
   \right)^{\tau}.
    \]
In the general case, one can only express $\widehat{T}_\alpha$ in terms of $t_n$ if one computes certain preimages (see Remark~\ref{rem:preimage}).

We show that these new pairings are non-degenerate in most cases.  
The pairings are amenable to efficient computation, for example for cryptographic purposes (see Algorithm~\ref{alg:alg}).  The algorithm is essentially the same as Miller's algorithm, adapted to the sesquilinear situation \cite{Miller}.
The new pairings presented here have already been applied to isogeny-based cryptography \cite{MaculaStange, GalbraithGilchristRobert}.

Both the Tate-Lichtenbaum pairing and Weil pairing have a wide variety of interpretations in terms of cohomology, intersection pairings, Cartier duality, etc.  In this paper we take an elementary approach in terms of divisors.  However, the new pairings were discovered while revisiting an interpretation of these pairings in terms of the monodromy of the Poincar\'e biextension studied in the author's PhD thesis \cite{thesis}.  A companion paper will explain these new pairings in that context, and their relationship with elliptic nets and height pairings.
 
\textbf{Notations.}\label{sec:notations}  Greek letters ($\alpha, \beta, \ldots$) generally refer to elements of the ring $R$, with the exception of $\sigma$, which is an element of a Galois group, and $\eta$ and $\epsilon$, which are maps in Section~\ref{sec:CM}.  Roman letters in lower case ($g, h, \ldots$) will generally refer to elements of $\GG_m$ (with the exception of $f$ and $g$, sometimes denoting functions), and capital roman letters (besides $R$ and $E$) typically refer to points of an elliptic curve $E$.  We use the exponent $\otimes_\ZZ R$ for the extension of scalars from $\ZZ$ to $R$ when viewing an abelian group in multiplicative notation as a $\ZZ$-module, as in $\GG_m^{\otimes_\ZZ R}$.  Simple tensors are written $g^{\otimes \alpha}$, but we will suppress the $\otimes$, writing $g^\alpha$.  Note, however, that we will continue to view this as a left $R$-module.  Regular exponents will be reserved for the module action of $R$ and $\ZZ$ when in a multiplicative notational mode.  In particular, we have the slightly counter-intuitive\footnote{We opted for this slight dissonance over the available alternatives, which were a switch to additive notation in the multiplicative group, or the use of notation $^\beta(^\alpha x) = \,^{\beta\alpha}x$.}  
\[
   (x^{ \alpha})^\beta = x^{ \beta \alpha}.
\]
For this reason we write $(\GG_m^{\otimes_\ZZ R})^{R\alpha}$ for the image of the multiplicative left $R$-module $\GG_m^{\otimes_\ZZ R}$ under the action of the $R$-submodule $R\alpha$, or equivalently, under $R \alpha R$.  We refer to this as the set of \emph{$\alpha$-powers} of $\GG_m^{\otimes_\ZZ R}$.  (If $\alpha \in \ZZ$, or more generally the centre of $R$, we can simplify the notation from $(\GG_m^{\otimes_\ZZ R})^{R\alpha}$ to $(\GG_m^{\otimes_\ZZ R})^{\alpha}$.)

We denote the algebraic closure of a field $K$ by $\overline{K}$.  We denote the action of an endomorphism $\alpha \in R$ on $P \in E$ by $[\alpha]P$.  For an $R$-module $M$, write $M[\alpha] := \{ m \in M : \alpha m = 0 \}$.  When $R$ is commutative, this is again an $R$-module.

\textbf{Acknowledgements.}  The author is grateful to Damien Robert for rekindling her interest through his recent work \cite{Robert},\cite{Robert2}, his interest in the author's thesis, and several generous discussions, which inspired this work.  A special debt is due to Joseph Macula and Damien Robert for corrections on an earlier draft.  The author also thanks Joseph H. Silverman and Drew Sutherland for helpful feedback, and an anonymous referee for a careful reading and very useful comments.

 \section{Classical pairings}
 \label{sec:classical}

\subsection{The Weil pairing}
\label{sec:backweil}

   This section follows Miller \cite{Mil2} and Silverman \cite[Chap III, \S8]{Sil1}.  For the more general Weil pairing, see \cite{Garefalakis}, \cite[Exercise III.3.15]{Sil1}.  

\begin{definition}[Weil pairing: first definition]
\label{defn: weil1}\index{Weil pairing}
Let $m > 1$ be an integer.  Let $E$ be an elliptic curve defined over a field $K$ which contains the field of definition of $E[m]$, and with characteristic coprime to $m$ in the case of positive characteristic.  Suppose that $P, Q \in E[m]$.  Choose divisors $D_P$ and $D_Q$ of disjoint support such that
\[
D_P \sim (P) - (\mathcal{O}), \qquad D_Q \sim (Q) - (\mathcal{O}).
\]
Then $mD_P \sim mD_Q \sim 0$, hence there are functions $f_P$ and $f_Q$ such that
\[
\div(f_P) = mD_P, \qquad \div(f_Q) = mD_Q.
\]
The Weil pairing
\[
e_m: E[m] \times E[m] \rightarrow \mu_m
\]
is defined by
\[
e_m(P,Q) = \frac{f_P(D_Q)}{f_Q(D_P)}.
\]
\end{definition}

For example, we can choose $D_P$ and $D_Q$ disjoint as follows: first choose some $T$ such that $T \not\in \{ \mathcal{O}, -P, Q, Q-P \}$.  Then set $D_P = (P+T) - (T)$ and $D_Q = (Q) - (\mathcal{O})$.  Set the notation $f_{m,X}$ for the rational function with divisor $m(X) - m(\mathcal{O})$.  Then,
\begin{equation*}
e_m(P,Q) = \frac{f_P(D_Q)}{f_Q(D_P)} = \frac{f_P(Q)f_Q(T)}{f_P(\mathcal{O})f_Q(P+T)} = \frac{f_{m,P}(Q-T)f_{m,Q}(T)}{f_{m,P}(-T)f_{m,Q}(P+T)}.
\end{equation*}

\begin{definition}[Weil pairing: second definition]
\label{defn: weil2}
Let $\phi: E \rightarrow E'$ be an isogeny between elliptic curves defined over a perfect field $K$ which contains the field of definition of $\ker(\phi)$ and $\ker(\widehat{\phi})$, and with characteristic coprime to $\deg \phi$ in the case of positive characteristic.  Suppose that $P \in \ker \widehat{\phi}$, and $Q \in \ker {\phi}$.
Let $g_P$ be a rational function with principal divisor
\[
\div(g_P) = \phi^*( (P) - (\mathcal{O}) ).
\]
(In the case that $\phi = [m]$, this implies
$g_P^m = f_{m,P} \circ [m]$.)
The Weil pairing
\[
e_\phi: \ker \widehat{\phi} \times \ker \phi \rightarrow \mu_m
\]
where $m$ is any positive integer with $\ker \phi \subseteq E[m]$, and $\mu_m$ denotes the $m$-th roots of unity,
is defined by
\[
e_\phi(P,Q) = \frac{g_P(X+Q)}{g_P(X)},
\]
where $X$ is any auxiliary point chosen disjoint from the supports of $g_P$ and $g_P \circ t_Q$ (the function $g_P$ precomposed with translation by $Q$).  
\end{definition}

Taking the isogeny $\phi$ to be the multiplication-by-m map $[m]$ recovers the $m$-Weil pairing as in the first definition.

The standard properties are as follows.

\begin{proposition}
\label{prop: weilprop}
Suppose $m$ is coprime to $\operatorname{char}(K)$ in the case of positive characteristic.  Definitions \ref{defn: weil1} and \ref{defn: weil2} are well-defined, equal when defined, and have the following properties (where defined in the case of the first definition):
\begin{enumerate}
\item Bilinearity: for $\phi$ an isogeny, $P, P_1, P_2 \in \ker \widehat{\phi}$ and $Q, Q_1, Q_2 \in \ker \phi$,
\begin{align*}
e_\phi(P_1+P_2,Q) &= e_\phi(P_1,Q)e_\phi(P_2,Q), \\
e_\phi(P,Q_1+Q_2) &= e_\phi(P,Q_1)e_\phi(P,Q_2).
\end{align*}
\item Alternating: for $P \in E[m]$,
\[
e_m(P,P)=1.
\]
\item Skew-symmetry: for $\phi$ an isogeny, $P \in \ker \widehat{\phi}$ and $Q \in \ker \phi$,
\[
e_\phi(P,Q) = e_{\widehat{\phi}}(Q,P)^{-1}.
\]
\item Non-degeneracy: for nonzero $P \in E[m](\overline K)$, there exists $Q \in E[m](\overline K)$ such that
\[
e_m(P,Q) \neq 1.
\]
\item Coherence: for $\phi, \psi$ isogenies such that $\psi \circ \phi$ is well-defined, for $P \in \ker \widehat{\phi} \circ \widehat{\psi}$, and $Q \in \ker \phi$,
\[
e_{\psi \circ \phi}(P,Q) = e_\phi(\widehat{\psi}P,Q).
\]
and for 
$P \in \ker \widehat{\psi}$, and $Q \in \ker \psi \circ \phi$,
\[
e_{\psi \circ \phi}(P,Q) = e_\psi(P,\phi Q).
\]
\item Compatibility:  for $\phi: E \rightarrow E'$ an isogeny, and $m$-torsion points $P \in E'[m]$ and $Q \in E[m]$,
\[
e_m(\widehat{\phi}P, Q) = e_m(P, \phi Q).
\]
\item Galois invariance: for $P, Q \in E[m]$, and $\sigma \in \operatorname{Gal}(\overline K/K)$,
\[
e_m(P,Q)^\sigma = e_m(P^\sigma, Q^\sigma).
\]
\end{enumerate}
\end{proposition}

\begin{proof}
For example, see \cite[Chapter 16]{thesis}, \cite{Robert}, \cite[Sec 3.1]{classactionpairing}.
\end{proof}

    For elliptic curves over $\CC$, the Weil pairing can be interpreted as a determinant, or an intersection pairing; see \cite{GalC}.
The Weil pairing also arises from the Cartier duality of the kernels of an isogeny and its dual; see Mumford \cite[IV.\S20, p.183-5]{Mum} and Milne \cite[\S11,16]{MilStorrs}.

\subsection{The Tate-Lichtenbaum pairing}
\label{sec:backtate}

Another pairing intimately related to the Weil pairing is the Tate-Lichtenbaum pairing.  This pairing was first defined by Tate \cite{Tat} for abelian varieties over $p$-adic number fields in 1958.  In 1959, Lichtenbaum defined a pairing on Jacobian varieties and showed that it coincided with the pairing of Tate \cite{Lic}.  The pairing was introduced to cryptography by Frey and R\"uck \cite{FreyRuck}.  Descriptions can be found in Silverman \cite[VIII.2, X.1]{Sil1} and Duquesne-Frey \cite{DuqFre2}.
For our version here, see for example \cite{Gal}.

\begin{definition}
\label{defn: tate}
Let $m > 1$ be an integer.  Let $E$ be an elliptic curve defined over a field $K$.  Suppose that $P \in E(K)[m]$.  Choose divisors $D_P$ and $D_Q$ of disjoint support such that
\[
D_P \sim (P) - (\mathcal{O}), \qquad D_Q \sim (Q) - (\mathcal{O}).
\]
Then $mD_P \sim 0$, hence there is a function $f_P$ such that
\[
\div(f_P) = mD_P.
\]
The Tate-Lichtenbaum pairing
\[
t_m: E(K)[m] \times E(K)/mE(K) \rightarrow K^* / (K^*)^m
\]
is defined by
\[
t_m(P,Q) = f_P(D_Q).
\]
\end{definition}

\begin{proposition}
  \label{prop:tateprops-classical}
Definition \ref{defn: tate} is well-defined, and has the following properties:
\begin{enumerate}
\item Bilinearity: for $P,P' \in E(K)[m]$ and $Q,Q' \in E(K)$
\begin{align*}
t_m(P+P',Q) &= t_m(P,Q)t_m(P',Q), \\
t_m(P,Q+Q') &= t_m(P,Q)t_m(P,Q').
\end{align*}
\item Non-degeneracy: Let $K$ be a finite field containing the $m$-th roots of unity $\mu_m$.  For nonzero $P \in E(K)[m]$, there exists $Q \in E(K)$ such that
\[
t_m(P,Q) \neq 1.
\]
Furthermore, for $Q \in E(K) \backslash mE(K)$, there exists $P \in E(K)[m]$ such that
\[
t_m(P,Q) \neq 1.
\]
\item Compatibility:  for an isogeny $\phi: E \rightarrow E'$, an $m$-torsion point $P \in E'$ and a point $Q \in E$,
\[
t_m(\widehat{\phi}P, Q) = t_m(P, \phi Q).
\]
\item Galois invariance: for $P, Q \in E[m]$, and $\sigma \in \operatorname{Gal}(\overline K/K)$,
\[
t_m(P,Q)^\sigma = t_m(P^\sigma, Q^\sigma).
\]
\end{enumerate}
\end{proposition}

\begin{proof}
See for example \cite[Chapter 16]{thesis}, \cite{Robert} and \cite[Sec 3.2]{classactionpairing}.
\end{proof}

\begin{remark}
	\label{rem:red}
For purposes such as cryptography, where $K = \FF_q$ and we wish to compare values of the Tate-Lichtenbaum pairing, it is typical to apply a final exponentiation by $(q-1)/m$ in order to obtain values in $\mu_m$.

Including this final exponentiation, there is a more general notion of Tate pairing associated to a $\FF_q$-rational isogeny $\phi: E \rightarrow E'$, that is,
\[
t_\phi: \ker \widehat{\phi}(\FF_q) \times E'(\FF_q)/ \phi E(\FF_q) \rightarrow \mu_m,
\]
where $m$ is any positive integer so that $\ker \phi \subseteq E[m] \subseteq E[q-1]$.  This generalizes the definition above when $\phi = [m]$, and can be given by
\[
t_\phi(P,Q) = e_{\phi}( \pi_q(T)-T, P),
\]
where $T$ is an arbitrarily chosen $\phi$-preimage of $Q$, $\pi_q$ is the $q$-power Frobenius, and $e_\phi$ is the Weil pairing.  It has the property that its values agree with those of $t_m^{\frac{q-1}{m}}$ on the common codomain; in other words, it is a restriction.  See \cite{Bruin}, \cite{Robert} and \cite[Sec 3.2]{classactionpairing}; see also \cite{Garefalakis}.
\end{remark}

\section{The calculus of $R$-divisors}

Let $R$ be an order in an imaginary quadratic field or quaternion algebra.  We wish to extend scalars from the divisor group $\Div(E)$, considered as a $\ZZ$-module, to the $R$-module $R \otimes_\ZZ \Div(E)$.   The purpose of this section is to verify that all usual notions (divisor, principality, pullback and pushforward, divisor of a function, evaluation of a function at a divisor, Weil reciprocity, etc.) are compatible, defined, and well-behaved under this extension.

Throughout the rest of the paper, we choose an integral basis:  write $R = \ZZ[\tau_i] :=  \sum_i \tau_i \ZZ$, where $\tau_0 = 1$ and we let $i$ range in $\{0,1\}$ or $\{0,1,2,3\}$ according to the rank $r \in \{2,4\}$ of $R$. When we sum over $i$ the range will be understood in context.  

Such a ring $R$ comes equipped with an involution which we term \emph{conjugation}, denoted $\alpha \mapsto \overline{\alpha}$.  In the quaternion algebra case, this is order reversing:  $\overline{\alpha \beta} = \overline \beta \overline \alpha$.

Let $E$ be an elliptic curve with divisor group $\Div(E)$.
We extend common notions from $\Div(E)$ to $R \otimes_\ZZ \Div(E)$.  We emphasize that in this section we make no assumption that $E$ has complex multiplication.

\subsection{$R$-divisors}

We define $\Div_R(E) := R \otimes_\ZZ \Div(E)$ to be the $R$-module generated by all symbols $(P)$, where $P$ is a point of $E$, i.e. finite formal $R$-linear combinations $\sum_P \alpha_P (P), \alpha_P \in R$ of such symbols, which we call \emph{$R$-divisors}.  
(We will frequently suppress the $\otimes$ for notational simplicity.)
Then $\Div_R(E)$ is an $R$-module under the action $\alpha \cdot (\beta \otimes D) = \alpha \beta \otimes D$.  A divisor $\sum_P \alpha_P (P)$ is of degree $0$ if $\sum_P \alpha_P = 0$ in $R$; these form a sub-$R$-module $\Div^0_R(E) \cong R \otimes_\ZZ \Div^0(E)$.  

In the presence of a preferred integral basis $\tau_i$ for $R$, we can write any $R$-divisor uniquely as a sum over $i$:
\[
\sum_P\left(\sum_i m_{i,P}\tau_i\right)(P) = \sum_i \tau_i \left(\sum_P m_{i,P} (P) \right).
\]
We say that an $R$-divisor is \emph{principal} if it is an $R$-linear combination of principal divisors of $\Div(E)$ (in which case it is certainly of degree zero).
We see that the principal divisors form a sub-$R$-module and we define $\Pic_R(E)$ and $\Pic_R^0(E)$ to be the $R$-module quotient of $\Div_R(E)$ and $\Div_R^0(E)$ by the principal divisors.  We use $\sim$ to denote linear equivalence (equivalence up to principal divisors).  Observe that 
$\Pic_R(E) \cong R \otimes_\ZZ \Pic(E)$, $\Pic_R^0(E) \cong R \otimes_\ZZ \Pic^0(E)$.

Recall from the introduction that we use the notation $G^{\otimes_\ZZ R}$ for the extension of scalars from $\ZZ$ to $R$ for a $\ZZ$-module written in multiplicative notation (i.e. a group $G$ in multiplicative notation).  
Let $\GG_m$ be the multiplicative group.  Then $\GG_m^{\otimes_\ZZ R}$ is an $R$-module whose action is written multiplicatively as $\alpha \cdot x = x^{\otimes \alpha} = x^\alpha$.  
As a reminder, the action is still a left action, so
\[
	\left({\prod g_i^{\tau_i}}\right)^{\alpha} = \prod g_i^{ {\alpha\tau_i}}.
\]
It also has a conjugation which will be useful:
\[
	\overline{\prod g_i^{\tau_i}} := \prod g_i^{ \overline{\tau_i}}.
\]

Similarly, the unit group of the function field, $K(E)^*$, extends to $(K(E)^*)^{\otimes_{\ZZ}R}$, and we may write, for example, $f^\alpha$ for $f \in K(E)^*$ acted upon by $\alpha \in R$.  Observe that these definitions are compatible with evaluation of a function at a point, i.e. we can define
\[
  (f^\alpha)(P) := (f(P))^\alpha, \quad f \in K(E)^*, P \in E(K),
\]
and $(fg)(P) := f(P)g(P)$ for $f,g \in (E(K)^*)^{\otimes_\ZZ R}$, at which point evaluation at $P$ becomes an $R$-module homomorphism from $(K(E)^*)^{\otimes_\ZZ R}$ to $(K^*)^{\otimes_\ZZ R}$.

We extend the notion of the divisor of a function $R$-linearly also, defining
\begin{equation}
\label{eqn:conjf}
  \div( f^\alpha ) := \alpha \cdot \div(f), \quad f \in K(E)^*, \alpha \in R,
\end{equation}
and $\div(fg) := \div(f) + \div(g)$ for $f,g \in (E(K)^*)^{\otimes_\ZZ R}$,
so that $\div$ becomes an $R$-module homomorphism.
Thus principal divisors are those which are divisors of $f \in (K(E)^*)^{\otimes_{\ZZ}R}$.

We define the usual push-foward and pull-back operations on divisors by extending $R$-linearly.  Suppose $\phi: E \rightarrow E'$.  Then
\[
	\phi^*\left( \alpha D \right) := \alpha \phi^* D, \quad
	\phi_*\left(  \alpha D \right) :=  \alpha \phi_* D.
\]
These inherit the usual desired properties:
\begin{enumerate}
	\item $\phi_* \phi^* D =  (\deg \phi) D$
	\item $\phi^* \div(f) = \div( \phi^* f)$, $\phi_* \div(f) = \div( \phi_* f)$
	\item $ (\phi \circ \psi)_*=\phi_* \psi_* $, $ (\phi \circ \psi)^*=\psi^* \phi^*$
\end{enumerate}
where we define $\phi_* (f^\alpha) := (\phi_* f)^\alpha$ and $\phi^*( f^\alpha ) := (\phi^* f)^\alpha$.

We also have a Galois action: 
$\left( \alpha D \right)^\sigma := \alpha (D^\sigma)$ for $\sigma \in \operatorname{Gal}(\overline{K}/K)$.

For a divisor $D = \sum n_P  (P) \in \Div(E)$, $n_P \in \ZZ$, we define
\[
D^\Sigma := \sum [n_P]P \in E.
\]

Viewing $E$ as a $\ZZ$-module, we obtain an $R$-module $R \otimes_\ZZ E$.
Then we have an $R$-module isomorphism
\[
\Pic_R^0(E) \cong R \otimes_\ZZ E, \quad
\alpha D \mapsto \alpha \otimes D^\Sigma.
\]
To show this is an isomorphism, we need to check that it is injective (surjectivity is clear).  If $D = \sum_i \tau_i D_i \mapsto \mathcal{O}$ 
then $D_i^\Sigma = \mathcal{O}$ for all $i$, so $D$ is principal.  In fact, an inverse is given by 
\[
	\sum_i \tau_i \otimes P_i \mapsto \sum_i \tau_i ( (P_i) - (\mathcal{O}) ).
\]

\subsection{Evaluation of functions at divisors}

We define evaluation of $f^\alpha$ for $f \in K(E)$, $\alpha \in R$ at $D \in \Div(E)$ as
\[
  (f^\alpha)(D) := (f(D))^\alpha,
\]
and extend to $\Div_R(E)$ by defining for $D \in \Div(E)$, $f \in (K(E)^*)^{\otimes_\ZZ R}$,
\[
f(\alpha \cdot D) := f(D)^{\overline{\alpha}}.
\]
This definition requires that the supports of $D$ and $\div(f)$ are disjoint.  
Observe the vinculum\footnote{Thank you to my brother and Wikipedia for teaching me this term for an \texttt{$\backslash\text{overline}$}.}, which reflects the duality between $f$ and $D$.  Among other things, it allows for the two left $R$-actions to interact as follows in the non-commutative setting:
\[
	f(\alpha \beta \cdot D) = f(\beta \cdot D)^{\overline{\alpha}} = f(D)^{\overline \beta \overline\alpha} = f(D)^{\overline{\alpha\beta}}.
\]

\subsection{Weil reciprocity}
A variation of Weil reciprocity (\cite[Chapter VI, Corollary to Theorem 10]{Lang}) holds for us:
\begin{theorem}
\label{thm: weilrecipgen}
    Let $f, g \in (K(E)^*)^{\otimes_\ZZ R}$ have disjoint support.  Then
    \[
	    f(\div(g)) = \overline{g(\div(f))}.
    \]
\end{theorem}

\begin{proof}
	The proof relies on Weil reciprocity for $\Div(E)$.  Suppose $f = \prod_i f_i^{ \tau_i}$ and $g = \prod_j g_j^{ \tau_j}$.  We have
	\begin{align*}
		f(\div(g))
		&= \prod_i f_i(\div(g))^{ \tau_i} 
		= \prod_{ij} f_i(\div(g_j))^{ \overline{\tau_j} \tau_i}
		= \prod_{ij} g_j(\div(f_i))^{ \overline{\tau_j} \tau_i} \\
		&= \overline{\prod_{ij} g_j(\div(f_i))^{ \overline{\tau_i} \tau_j}} 
		= \overline{\prod_{j} g_j(\div(f))^{ \tau_j}} 
		= \overline{ g(\div(f)) }.
	\end{align*}
\end{proof}

\section{Sesquilinear pairings}
\label{sec:sesq}

If $R$ is commutative, an \emph{$R$-sesquilinear pairing}, conjugate linear on the right, is a bilinear pairing
$\langle \cdot, \cdot \rangle$ on a pair of $R$-modules, taking values in another $R$-module, that satisfies
\[
\langle \alpha x , \beta y \rangle = { \overline\beta}\alpha \cdot \langle x,y \rangle, \text{ for all } \alpha, \beta \in R.
\]
We say instead that it is conjugate linear on the left when $\langle \alpha x , \beta y \rangle = {\beta}\overline\alpha \cdot \langle x,y \rangle$.
For the non-commutative case, we need to add a type of twisting.
Recall that $R$ is a maximal order in a division algebra.  Thus we can set the notation $R_\gamma := \gamma^{-1} R \gamma \cap R$, a subring of $R$.
For $\gamma \in R$ and $\delta \in R_\gamma$, let $\delta^{(\gamma)}$ be defined as that element of $R$ which satisfies $\delta^{(\gamma)} \gamma = \gamma \delta$.
For us, a \emph{$\gamma$-twisted $R$-sesquilinear pairing} is a bilinear pairing 
$\langle \cdot, \cdot \rangle$ on a pair of modules, the first an $R_\gamma$-module and the second an $R$-module, taking values in another $R$-module, that satisfies
\[
	\langle \alpha x , \beta y \rangle = { \overline\beta}\;{{\alpha}^{(\overline\gamma)}} \cdot \langle x,y \rangle, \text{ for all } \alpha \in R_\gamma, \beta \in R.
\]
Observe that for rank $2$, commutativity implies $\delta^{(\gamma)} = \delta$ and $R_{\gamma} = R$, so the $\gamma$-twisting is vacuous, and we recover sesquilinear pairings in the traditional sense.

The purpose of this section is to generalize the definitions of the classical Weil and Tate pairings (Section~\ref{sec:classical}) in the context of $R$-divisors, to obtain sesquilinear pairings, and prove they enjoy the same host of properties, suitably adapted.  The proofs of the standard properties are straightforward, although finicky, particularly in the case of rank four.  But the proof of non-degeneracy for these pairings is non-trivial (as it is in the classical case).

\subsection{Generalization of Tate-Lichtenbaum pairing}
\label{sec:genTate}

For each $\alpha \in R$, we define an $\alpha$-twisted $R$-sesquilinear pairing, conjugate linear on the right, generalizing the Tate-Lichtenbaum pairing:  
\[
	T_\alpha : \Pic_R^0(E)[\overline{\alpha}] \times \Pic_R^0(E) / R\alpha \Pic_R^0(E) \rightarrow \GG_m^{\otimes_\ZZ R} /  (\GG_m^{\otimes_\ZZ R})^{R\overline{\alpha}},
\]
by
\[
T_\alpha(D_P,D_Q) := f_P(D_Q) 
\quad \text{ where }\quad
\div(f_P) = \overline{\alpha} \cdot D_{P},
\]
where $D_P$ and $D_Q$ are chosen to have disjoint support.  Observe that $\Pic_R^0(E)[\overline{\alpha}]$ is an $R$-module when $R$ is commutative, but in general we can only assume it is an $R_{\overline{\alpha}}$-module.
Also, we use $R\alpha \Pic_R^0(E)$ since $\alpha \Pic_R^0(E)$ may not be an $R$-module in the non-commutative case.
Finally, the target could equivalently be written $\GG_m^{\otimes_\ZZ R} /  (\GG_m^{\otimes_\ZZ R})^{R\overline{\alpha}R}$.

Although the notation $T_\alpha$ does not reference $R$, its definition does depend upon the choice of $R$ containing $\alpha$.  For example, $T_\alpha$ for $R$ of rank $4$ does not agree with $T_\alpha$ defined for a rank two subring containing $\alpha$; this is a phenomenon similar to the relationship between $T_n$ and $t_n$ described in Proposition~\ref{prop:tate-n}.  Even $R \subseteq R'$ of the same rank can result in different pairings.  In this paper, we are assuming $R$ to be fixed, being either an imaginary quadratic or quaternion order.

Recall our convention that $R = \ZZ[\tau_i] :=  \sum_i \tau_i \ZZ$, where $\tau_0 = 1$ and we let $i$ range in $\{0,1\}$ or $\{0,1,2,3\}$ according to the rank $r \in \{2,4\}$ of $R$.  In the rank $2$ case, we will write $\tau := \tau_1$ for simplicity.  
To satisfy the condition on supports, observe that for any divisor $D \in \Pic_R^0(E)$, there exist points $P_0,\ldots, P_{r-1} \in E$ so that
\begin{equation}
    \label{eqn:DP}
    D \sim  \sum_i \tau_i ( (P_i + S) - (S))
\end{equation}
for any auxiliary point $S \in E$.  
In particular, if $P_0, \ldots, P_{r-1}$ are such that $D_P \sim \sum_i \tau_i ( (P_i) - (\mathcal{O}) )$, and
\[
	\overline{\alpha} \tau_i = \sum_j \alpha_{ji} \tau_j,
\]
then we can take $f_P = \prod_i f_i^{ \tau_i} \in (K(E)^*)^{\otimes_\ZZ R}$, where
\begin{equation}
    \label{eqn:hpgp}
    \div(f_i) = \sum_{j=0}^{r-1} \alpha_{ij}(P_j) -\left(\sum_{j=0}^{r-1} \alpha_{ij}\right)(\mathcal{O}),
\end{equation}
and then by a judicious choice of $D_Q$ (choosing $S$ in the linearly equivalent form \eqref{eqn:DP}), we can satisfy the condition on disjoint supports.

\begin{remark}
\label{rem:miller}
The equations \eqref{eqn:hpgp} allow for a Miller-style algorithm to compute this pairing \cite{Miller} \cite[\S 26.3.1]{GalbraithBook}.  This is 
    polynomial time in the coefficients of the minimal polynomial of $\alpha$.  
For example, if $R$ has basis $1$ and $\tau$, and $D_P = \left( (P_0) - (\mathcal{O}) \right) + \tau \cdot \left( (P_1) - (\mathcal{O}) \right)$, and
\[
	\overline{\alpha} = a + c \tau, \quad \overline\alpha \tau = b  + d \tau, \quad a,b,c,d \in \ZZ,
\]
then $f_P = f_0 f_1^{ \tau} \in (K(E)^*)^{\otimes_\ZZ R}$, where
\begin{equation}
    \label{eqn:hpgp-2}
\div(f_0) = a(P_0) + b(P_1)-(a+b)(\mathcal{O}), \quad
\div(f_1) = c(P_0) + d(P_1) - (c+d)(\mathcal{O}).
\end{equation}
More details are given for the CM case in Algorithm~\ref{alg:alg}.
\end{remark}

   \begin{theorem}
\label{thm:tateprops}
    The pairing defined above is well-defined, bilinear, and satisfies
    \begin{enumerate}
	    \item Twisted sesquilinearity:    For $\gamma \in R_{\overline\alpha}$ and $\delta \in R$, \[
    {T}_\alpha(\gamma \cdot D_P, \delta \cdot D_Q)
    = {T}_\alpha(D_P,D_Q)^{{\overline\delta}\;{\gamma}^{(\overline{\alpha})}}.
    \]
        \item Compatibility:   
            Let $\phi: E \rightarrow E'$.   Then 
    \[
    T_\alpha(\phi_* D_P,\phi_* D_Q) = T_\alpha(D_P,D_Q)^{\deg \phi}.
    \]
    \item Coherence:
	    Suppose $D_P \in \Pic_R^0(E)[\overline{\beta \alpha}]$, and $D_Q \in \Pic_R^0(E)/ R \beta\alpha  \Pic_R^0(E)$.  Then
    \[
    T_{\beta \alpha}(D_P, D_Q) \bmod{ (\GG_m^{\otimes_\ZZ R})^{R\overline\alpha}}
    = T_{\alpha}( \overline{\beta} \cdot D_P, D_Q \bmod R\alpha  \Pic_R^0(E) ).
    \]
    Suppose 
    $D_P \in \Pic_R^0(E)[\overline{\beta}]$, and $D_Q \in \Pic_R^0(E)/ R \beta\alpha  \Pic_R^0(E)$.  Then
    \[
    T_{\beta\alpha}(D_P, D_Q) \bmod{ (\GG_m^{\otimes_\ZZ R})^{R\overline\beta}}
    = T_{\beta}( D_P, \alpha \cdot D_Q \bmod R\beta \Pic_R^0(E) ).
    \]
    \item Galois invariance: Suppose $E$ is defined over a field $K$.
 Let $\sigma \in \operatorname{Gal}(\overline K/K)$.  Then
\[
T_\alpha(D_P,D_Q)^\sigma = T_{\alpha}(D_P^\sigma, D_Q^\sigma).
\]
    \end{enumerate}

\end{theorem}

\begin{proof}
\emph{Choice of representative $D_Q$ in the divisor class:}
Suppose $D_Q \sim D_Q'$.  Then for some $g \in (K(E)^*)^{\otimes_\ZZ R}$, having divisor $\div(g) = D_Q - D_Q'$, and using Weil reciprocity\footnote{There's a subtlety here.  Observe that $\overline{ (g^\beta)^\alpha }=\overline{g^{\alpha\beta}} = g^{\overline{\alpha\beta}} = g^{\overline\beta \; \overline\alpha} = g^{\overline{\alpha}^{(\overline\beta)} \overline{\beta}} = (g^{\overline{\beta}})^{\overline\alpha^{(\overline\beta)}}$, so that it is only in the case that $R$ is commutative that $\overline{g^\alpha} = \overline{g}^{\overline\alpha}$.  However, it is still true that $\overline{g(D_P)^\alpha} \in (\GG_m^{\otimes_\ZZ R})^{R\overline\alpha}$. } (Theorem~\ref{thm: weilrecipgen}),
\[
	f_P(D_Q)f_P(D_Q')^{-1} 
=
f_P(\div(g)) = \overline{g(\div(f_P))}
= \overline{g( \overline{\alpha} \cdot D_P )}
= \overline{g(D_P)^{{\alpha}}}
\in 
(\GG_m^{\otimes_\ZZ R})^{R\overline\alpha}.
\]

\emph{Choice of $D_Q$ modulo $R\alpha\Pic_R^0(E)$:}
\[
f_P(D_Q + \gamma \alpha \cdot D')
=
f_P(D_Q) f_P(D')^{\overline{\alpha}\;\overline\gamma}.
\]

\emph{Choice of representative $D_P$ in the divisor class:}
Suppose $D_P \sim D_P'$.  Notice that if we let $\div(f_P) =\overline{\alpha} \cdot D_P$ and $\div(f_P') = \overline{\alpha} \cdot D_P'$, then 
\[
\div(f_P') = \div(f_P) + \overline{\alpha} \cdot (D_P' - D_P).
\]
Hence 
$f_P' = f_P g^{\overline\alpha}$ where $\div(g) = D_P' - D_P$, 
which is principal by assumption.  Then
\[
f_P'(D_Q) = f_P(D_Q) g(D_Q)^{\overline\alpha}.
\]

\emph{Choice of $f_P$:}  Any two choices of $f_P$ differ by a constant scalar, but $D_Q$ has degree $0$ by assumption, so the constant cancels in the formula $f_P(D_Q)$.

\emph{Bilinearity:}
Let $D_P$, $D_P' \in \Div^0_R(E)[\overline{\alpha}]$ and $\div(f_P) = \overline{\alpha} \cdot D_P$, $\div(f_P') = \overline{\alpha} \cdot D_P'$.  Then
\[
{T}_\alpha(D_P + D_P',D_Q) = f_P(D_Q) f_{P}'(D_Q) = {T}_\alpha(D_P,D_Q)  {T}_\alpha(D_P',D_Q).
\]
In the other factor,
\[
{T}_\alpha(D_P,D_Q+D_Q') =
f_P(D_{Q} + D_Q') = f_P(D_Q)f_P(D_Q') =
{T}_\alpha(D_P,D_Q)  {T}_\alpha(D_P,D_Q').
\]

\emph{Twisted sesquilinearity:}
Suppose $f_P$ has divisor $\overline{\alpha} \cdot D_P$.
In evaluating $T_\alpha(\gamma \cdot D_P, \delta \cdot D_Q)$, we evaluate the function with divisor
$\overline{\alpha} \cdot \gamma \cdot D_P = \gamma^{(\overline{\alpha})} \cdot \overline{\alpha} \cdot D_P$ at the divisor $\delta \cdot D_Q$.  Since $\div(f_P^{{\mu}}) = \mu \cdot \div(f_P)$ by \eqref{eqn:conjf}, this becomes
\[
	f_P( \delta \cdot D_Q)^{{\gamma}^{(\overline\alpha)}} = f_P(D_Q)^{\overline\delta\;{\gamma}^{(\overline\alpha)}}.
\]

\emph{Compatibility:}
    Observe that $\overline{\alpha} \cdot \phi_* D_P = \phi_* (\overline{\alpha} \cdot D_P)$.  Therefore, in the computation of $T_\alpha(\phi_* D_P, \phi_* D_Q)$, we evaluate $\phi_* f_{P}$ at $\phi_* D_Q$.  We have
    \[
	    \phi_* f_{P} (\phi_* D_{Q}) = f_P (\phi^* \phi_* D_Q) = f_P(D_Q)^{\deg \phi},
    \]
    where the last equality depends upon the fact that $\phi^* \phi_* D \sim (\deg \phi) D$ for $D \in \Pic^0_R(E)$.

    \emph{Coherence:}
    Both statements follow immediately from the definitions.

    \emph{Galois invariance:}  This is immediate, since by our definition of the actions of $R$ on the various entities involved, we have
    $(\gamma \cdot D)^\sigma = \gamma \cdot D^\sigma$ for any $\gamma \in R$.
\end{proof}

 \begin{remark}
   In cryptographic applications, we typically restrict to inputs defined over a field $\FF_q$.  If $R$ is commutative, to obtain canonical representatives of the codomain, it may be useful to post-compose with a map
\[
(\FF_q^*)^{\otimes_\ZZ R} / ((\FF_q^*)^{\otimes_\ZZ R})^{\overline\alpha} \rightarrow \mu_{\overline\alpha} := \{ u \in \mu_{N(\alpha)}^{\otimes_\ZZ R} \subseteq (\FF_q^*)^{\otimes_\ZZ R} : u^{\overline\alpha} = 1 \},
\]
given by
\[
x \mapsto x^{(q-1)\overline\alpha^{-1}}.
\]
\end{remark}

\begin{proposition}
\label{prop:tate-n}
    Let $n \in \ZZ$.  
For positive integers $n$, let
\[
t_n: E[n] \times E/[n]E \rightarrow \GG_m/\GG_m^n
\]
denote the usual Tate-Lichtenbaum pairing as in Section~\ref{sec:backtate}.
    Let 
    $D_P  \in \Pic_R^0(E)[n]$ and $D_Q  \in \Pic_R^0(E)$.
    Suppose
    \[
D_P \sim \sum_i \tau_i \cdot \left( (P_i) - (\mathcal{O})\right), \quad
D_Q \sim \sum_i \tau_i \cdot \left( (Q_i) - (\mathcal{O}) \right).
\]
    Then
    \[
	    T_n(D_P,D_Q) =
	    \prod_{i,j=0}^{r-1} t_n(P_i, Q_j)^{ \overline{\tau_j}\tau_i}.
	 \]
        Furthermore, when both of the following quantities are defined, we have
    \[
    {T}_{N(\alpha)}(D_P,D_Q) \equiv {T}_\alpha(D_P,D_Q)^\alpha \pmod{(\GG_m^{\otimes_\ZZ R})^{R\overline\alpha}}
    \]
\end{proposition}
\begin{proof}
By a linear equivalence, assume that
\[
D_P = \sum_i \tau_i \cdot \left( (P_i) - (\mathcal{O})\right), \quad
D_Q = \sum_j \tau_j \cdot \left( (Q_j+S) - (S) \right).
\]
where $S$ is chosen to avoid intersections of supports.
We have from \eqref{eqn:hpgp}, with $f_P = \prod_i f_i^{ \tau_i}$, that
\[
    \div(f_i) = n(P_i) - n(\mathcal{O}).
    \]
We obtain
    \[
	    T_n(D_P,D_Q) =
    \prod_{j} \left(\prod_i f_i( (Q_j + S) - (S) )^{ \tau_i} \right)^{\overline{\tau_j}}.
    \]
    That shows the first statement.  For the second,
    suppose $\div(f_P) = \overline\alpha \cdot D_P$.  Then for any divisor $D_Q$ with sufficiently disjoint support,
\[
(f_{P}^{{\alpha}})(D_Q) = f_{P}(D_Q)^{{\alpha}}.
\]
On the left, we see this is by definition a representative of $T_n(D_P,D_Q)$ in $\GG_m^{\otimes_\ZZ R}/(\GG_m^{\otimes_\ZZ R})^n$, since $\div(f_P^\alpha) = \alpha \cdot \div(f_P) = n D_P$.  However, looking at the right, this is also a representative of $T_\alpha(D_P,D_Q)^\alpha$ in $\GG_m^{\otimes_\ZZ R}/(\GG_m^{\otimes_\ZZ R})^{R\overline\alpha}$.
\end{proof}

    In particular, in the rank $2$ case (i.e. $\tau_0 = 1, \tau_1 = \tau$),
    \[
	    \overline{\tau} = \Tr(\tau) - \tau, \quad \overline{\tau}\tau = N(\tau),
    \]
    which gives (continuing the notation of Proposition~\ref{prop:tate-n}, in particular the definition of $P_i$, $Q_i$), 
    \begin{equation}
	    \label{eqn:Tn2}
     {T}_n (D_P,D_Q) = 
     \left( 
     t_n(P_0, Q_0)t_n(P_1,Q_1)^{N(\tau)} t_n(P_0,Q_1)^{\Tr(\tau)}\right)
     \left( t_n(P_1, Q_0) t_n(P_0,Q_1)^{-1}
     \right)^{\tau}.
    \end{equation}
    Let $\langle x,y \rangle$ be a bilinear pairing on $\ZZ[\tau]$.  Then
\begin{align*}
\langle x_1 + \tau x_2 , y_1 + \tau y_2 \rangle 
&:= 
\langle x_1, y_1 \rangle + N(\tau)\langle x_2, y_2 \rangle
+ \Tr(\tau) \langle x_1, y_2 \rangle + \tau \left( 
\langle x_2, y_1 \rangle - \langle x_1, y_2 \rangle
\right)
\end{align*}
defines a sesquilinear pairing (conjugate linear in second entry).  This explains the formula \eqref{eqn:Tn2}, and in fact we could define the pairing $T_n(D_P,D_Q)$ from $t_n(P_i,Q_i)$ directly by using Proposition~\ref{prop:tate-n} as a definition.

\begin{remark}
\label{rem:preimage}
There does not seem to be an analogous construction for $T_\alpha(D_P,D_Q)$ in terms of $t_n(P_i,Q_i)$.  The best we can do requires computing some preimages under multiplication maps.  Specifically, by coherence,
\[
	T_\alpha(D_P, \overline{\alpha} \cdot D_S) = T_n(D_P,D_S).
\]
To use this for calculation, letting $r=2$ (the commutative case) for simplicity, suppose  $D_S = (S_0) - (\mathcal{O}) + \tau \cdot \left( (S_1) - (\mathcal{O}) \right)$.  Then suppose
$\overline\alpha = a + c \tau, \overline\alpha \tau = b + d\tau$, $a,b,c,d \in \ZZ$.  Then
\begin{align*}
\overline{\alpha} \cdot D_S
&=
a(S_0) + b(S_1) - (a+b)(\mathcal{O}) + \tau \cdot \left(
c(S_0) + d(S_1) - (c+d)(\mathcal{O})
\right) \\
&\sim
([a]S_0 + [b]S_1) - (\mathcal{O})
+ \tau \cdot \left(
([c]S_0 + [d]S_1) - (\mathcal{O})
\right).
\end{align*}
Thus, we can give an expression for $T_\alpha(D_P, D_Q)$ in terms of the classical Tate-Lichtenbaum pairing applied to combinations of $P_0, P_1,S_0, S_1$ provided the $S_i$ solve
\[
[a]S_0 + [b]S_1 = Q_0, \quad
[c]S_0 + [b]S_1 = Q_1.
\]
\end{remark}

A principal ideal ring is one in which all right and left ideals are principal.

\begin{lemma}
	\label{lem:tatenon}
	Let $R$ be a ring with an involution called conjugation, $I$ be a principal two-sided ideal of $R$, and suppose that $R/I$ is a finite principal ideal ring.
	Let $t: A \times B \rightarrow R/I$ be a sesquilinear form on $R$-modules (conjugate linear in one variable).
	Suppose that $t$ is non-degenerate.  Then if $a \in A$ has annihilator $I$, then $t(a, \cdot)$ is surjective.  Furthermore, if $b \in B$ has annihilator $I$, then $t(\cdot, b)$ is surjective.
\end{lemma}

\begin{proof}
	Since $R' := R/I$ is a principal ideal ring, we claim that there is no proper $R$-submodule of $R'$ with annihilator $I$.  Indeed, every submodule $R''$ of $R'$ is cyclic as an $R'$ module, hence of the form $R'' \cong R'/ J$ for some ideal $J$ which is the annihilator of $R''$.  By a cardinality argument, if $R''$ is a proper submodule of $R'$, then $J$ is non-trivial and the annihilator of $R''$ as an $R$-module is strictly larger than $I$.

	Now let $a \in A$ have annihilator $I$.  Then $t(a,B)$ is an $R$-module with annihilator equal to the intersection of the annihilators of all elements $t(a,b) \in R/I$, $b \in B$.  If this intersection is equal to $I$, then we have surjectivity, by the preceding argument.  If not, then there exists some element $r \in R$ which does not annihilate $a$, but does annihilate $t(a,B)$.  These two properties, respectively, have the consequences that there exists $b \in B$ such that $t(ra,b) \neq 0$ by non-degeneracy, but simultaneously that $t(a,\overline{r} b) = 0$.  This contradiction completes the argument that $t(a,\cdot)$ is surjective.  The argument that $t(\cdot,b)$ is surjective is similar.
\end{proof}


\begin{theorem}
\label{thm:tatenon}
    Let $K$ be a finite field over which the endomorphisms of $R$ are defined.  Let $\alpha \in R$, such that $N(\alpha)$ is coprime to $char(K)$ and the discriminant of $R$.  Let $n = N(\alpha)$.  Suppose $K$ contains the $n$-th roots of unity.   Then 
    \[
T_\alpha : \Pic_R^0(E)[\overline{\alpha}](K) \times \Pic_R^0(E)(K) / R \alpha \Pic_R^0(E)(K) \rightarrow (K^*)^{\otimes_\ZZ R} /  ((K^*)^{\otimes_\ZZ R})^{R\overline\alpha}
\] 
is non-degenerate.  Furthermore, if $D_P$ has annihilator $R\overline\alpha R$, then $T_\alpha(D_P, \cdot)$ is surjective; and if $D_Q$ has annihilator $R \alpha R$, then  $T_\alpha(\cdot, D_Q)$ is surjective.
\end{theorem}

\begin{proof}
	First, a few preliminaries.
	Using the fact that $K^*$ is cyclic of order divisible by $\overline{\alpha}$, the target $(K^*)^{\otimes_\ZZ R} / ((K^*)^{\otimes_\ZZ R})^{R\overline\alpha} \cong  R/R\overline\alpha R$ as $R$-modules, and this is finite. 
	We wish to apply Lemma~\ref{lem:tatenon}. 

	If $R$ is an imaginary quadratic order, then its quotient $R/\overline\alpha R$ is a principal ideal ring (since $N(\alpha)$ is coprime to the discriminant).

	If $R$ is an order in a quaternion algebra, then $R \otimes \QQ_p \cong M_2(\QQ_p)$ for $p$ not dividing the discriminant of $R$.  This implies, in particular, that $R / p^k R \cong M_2(\ZZ/p^k\ZZ)$, which is a principal ideal ring.  By assumption, $N(\alpha)$ is coprime to the discriminant. 
	For any prime $\overline\alpha$, the ring $R/ R\overline{\alpha}R$ is a quotient of such a ring, hence a principal ideal ring.  
	In general, $R / R \overline\alpha R$ is a product of principal ideal rings, hence a principal ideal ring.  

	So by Lemma~\ref{lem:tatenon}, it suffices to check non-degeneracy.
	Consider first the non-degeneracy of $T_n$, $n \in \ZZ$.
Let $D_P$ be given.  We show non-degeneracy on the left by finding $D_Q$ so that $T_n(D_P,D_Q)$ is non-trivial.
By Proposition~\ref{prop:tate-n}, and the non-degeneracy of the traditional Tate pairing $t_n$, we can choose $D_Q$ so that $T_n(D_P,D_Q)$ is non-trivial (e.g., provided $P_0 \neq \mathcal{O}$, choose $Q_i$, $i > 0$ to be $\mathcal{O}$ to simplify the condition).  This depends upon the following fact:  the image of $T_n$ is taken modulo $n$-th powers, hence a non-$n$-th power entry in one position of $\GG_m^{\otimes_\ZZ R}$ implies the element represents a non-trivial coset. Hence $T_n$ is left-non-degenerate.  An exactly similar argument shows $T_n$ is right-non-degenerate.

Now we consider general $\alpha$, with $n = N(\alpha)$.  Suppose $\div(f_P) = \overline\alpha \cdot D_P$.  Then for any divisor $D_Q$ with sufficiently disjoint support, as observed in the proof of Proposition~\ref{prop:tate-n},
\begin{equation}
	\label{eqn:tatenon}
(f_{P}^{{\alpha}})(D_Q) = f_{P}(D_Q)^{{\alpha}}.
\end{equation}
By non-degeneracy of $T_n$, fixing non-trivial $D_P \in \Pic_R^0(E)[\overline\alpha](K) \subseteq \Pic_R^0(E)[n](K)$, one may choose $D_Q \in \Pic_R^0(E)(K)$ so that $T_n(D_P,D_Q)$ is not an $n$-th power.  The expression \eqref{eqn:tatenon} is a representative of $T_n(D_P,D_Q)$, so is not an $n$-th power.  Therefore $f_{P}(D_Q)$ cannot be an $\overline\alpha$-power in $\GG_m^{\otimes_\ZZ R}$.  However, this is a representative of $T_\alpha(D_P,D_Q)$.  Therefore we have shown left non-degeneracy.

On the right, fix a non-trivial $D_Q \in \Pic_R^0(E)(K)/R\alpha \Pic_R^0(E)(K)$.  Choose $\beta \in \ZZ[\alpha]$ such that $(\alpha,\beta)=\ZZ[\alpha]$, and $m := \alpha \beta  \in \ZZ$ and $m$ divides $n$.  By coprimality, we may choose a lift $\beta\cdot D_Q' \in \Pic_R^0(E)(K)/Rm \Pic_R^0(E)(K)$ of $D_Q$.  We know there exists some $D_P \in \Pic_R^0(E)[m](K)$ so that $T_m(D_P,{D_Q'})$ is non-trivial, using the earlier case (since $m$ divides $n$).  
Consider the two quantities 
\[
	T_\alpha(D_P, D_Q), \quad 
	T_m(D_P,{D_Q'}). 
\]
Suppose $\div(f_P) = m D_P = \overline{\alpha} \cdot \overline{\beta}\cdot D_P$.  Then the quantity $f_P(D_Q') \in (K^*)^{\otimes_\ZZ R}$ is a representative of both of the two quantities just displayed, in their respective domains.  Since $T_m(D_P,D_Q')$ is not an $m$-th power in $(K^*)^{\otimes_\ZZ R}$, we observe that $T_\alpha(D_P, D_Q) = T_\alpha(D_P,D_Q')^{\overline\beta}$ is not a $m$-th power, so $T_\alpha(D_P,D_Q')$ is not an $\overline\alpha$ power.  By coprimality, $T_\alpha(D_P,D_Q) = T_\alpha(D_P,D_Q')^{\overline\beta}$ is not an $\overline\alpha$ power.
\end{proof}

\subsection{Generalization of Weil pairing}
Let $\GG_m^{\otimes_\ZZ R}[\overline\alpha] = \{ x \in \GG_m^{\otimes_\ZZ R} : x^{\overline\alpha} = 1^{\otimes 0} \}$, which\footnote{Keep in mind the multiplicative nature of our notation: $1^{\otimes \tau} = 1^{\otimes 1} = 1^{\otimes 0} = x^{\otimes 0}$, all representing the identity element of the $R$-module.} we might call the $\overline\alpha$-th roots of unity in $\GG_m^{\otimes_\ZZ R}$.
We can define a generalization of the Weil pairing
\[
W_{\alpha} : \Pic_R^0(E)[\overline{\alpha}] \times \Pic^0_R(E)[{\alpha}] \rightarrow \GG_m^{\otimes_\ZZ R}[\overline\alpha],\quad
W_\alpha(D_P, D_Q) := f_P(D_Q) \overline{f_Q(D_P)}^{-1},
\]
where $\div(f_P) = \overline{\alpha} \cdot D_P$ and $\div(f_Q) = {\alpha} \cdot D_Q$, where the pairs ($f_P$, $D_Q$) and ($f_Q$, $D_P$) have disjoint support; we reuse the notation from the definition of $T_\alpha$ (Section~\ref{sec:genTate}).

\begin{remark}
    Comparing to $T_\alpha$, we may wish to write
\[
	W_\alpha(D_P,D_Q) \stackrel{?}{=} T_\alpha( D_P, D_Q) {\overline{T_{\overline{\alpha}}(D_Q,D_P)}}^{-1},
\]
but a priori, this is not well-defined, because the validity of the equality depends on the correct choice of representative for the coset of $T_\alpha( D_P, D_Q)$ or $T_{\overline{\alpha}}(D_Q, D_P)$.
\end{remark}

\begin{theorem}
\label{thm:weilprops}
The definition above is well-defined, bilinear, and satisfies:
\begin{enumerate}
	\item Restricted Sesquilinearity:  For $\gamma, \delta$ such that $\gamma^{(\alpha)} = \gamma$ and $\delta^{(\overline\alpha)} = \delta$, we have 
\[
    {W}_\alpha(\gamma \cdot D_P, \delta \cdot D_Q)
    = {W}_\alpha(D_P,D_Q)^{\overline{\delta}\gamma}.
    \]
    \item Conjugate skew-Hermitianity:
    \[
    W_\alpha(D_P,D_Q) = \overline{W_{\overline{\alpha}}(D_Q,D_P)}^{-1}.
\]
 \item Compatibility:   
            Let $\phi: E \rightarrow E'$.   Then 
    \[
    W_\alpha(\phi_* D_P,\phi_* D_Q) = W_\alpha(D_P,D_Q)^{\deg \phi}.
    \]
\item Coherence: For $D_P \in \Pic_R^0(E)[\overline{\beta\alpha}]$, $D_Q \in \Pic^0_R(E)[{\beta\alpha}]$,
    \[
	    W_{\beta\alpha}(D_P,D_Q) = W_{\alpha}( \overline{\beta} \cdot D_P, D_Q) \in \GG_m^{\otimes_\ZZ R}[\overline\alpha], \quad
	    W_{\beta\alpha}(D_P,D_Q) =W_\beta( D_P, {\alpha} \cdot D_Q) \in \GG_m^{\otimes_\ZZ R}[\overline\beta].
    \]
\item Galois invariance: Suppose $E$ is defined over a field $K$.  Let $\sigma \in \operatorname{Gal}(\overline K/K)$; then
\[
W_\alpha(D_P,D_Q)^\sigma = W_\alpha(D_P^\sigma, D_Q^\sigma).
\]
\end{enumerate}
\end{theorem}

\begin{proof}
We begin with well-definition.  
Suppose
$D_Q \sim D_Q'$ and $D_P \sim D_P'$, 
and
let $\div(g_1) = D_Q - D_Q'$ and $\div(g_2) = D_P - D_P'$.
From Weil reciprocity, 
\[
	\frac{\overline{f_Q(D_P)}}{\overline{f_Q'(D_P)}} 
=	\overline{\left(\frac{f_Q}{f_Q'}\right)(D_P)}
= \overline{g_1(D_P)^{\alpha}} = \overline{g_1( \overline\alpha \cdot D_P)} = 
\frac{f_P(D_Q)}{f_P(D_Q')}.
\]
Therefore, $W_\alpha(D_P,D_Q) = W_\alpha(D_P,D_Q')$.  By a symmetrical argument, $W_\alpha(D_P,D_Q) = W_\alpha(D_P',D_Q')$.  Note that a scalar change of $f_P$ or $f_Q$ will cancel.
Thus $W_\alpha$ is well-defined taking values in $\GG_m^{\otimes_\ZZ R}$.  
The proof of bilinearity is as for $T_\alpha$ in Theorem~\ref{thm:tateprops}.
From the definition, observe that
$W_\alpha(D_P,0) = W_\alpha(0,D_Q) = 1$.  In particular, bilinearity implies the image is in $\GG_m^{\otimes_\ZZ R}[\overline\alpha]$.

The argument for sesquilinearity of $T_\alpha$ in the proof of Theorem~\ref{thm:tateprops} works equally well here, as does the argument for compatibility.
Conjugate skew-Hermitianity is exactly from the definition of $W_\alpha$.
For coherence, recall that $\overline{\alpha \beta} = \overline{\beta} \overline{\alpha}$ and apply the definitions.
Galois invariance follows as in Theorem~\ref{thm:tateprops}. 
\end{proof}

Analogously to Proposition~\ref{prop:tate-n}, for $W_n$, we can give an expression in terms of the classical Weil pairing.

\begin{proposition}
\label{prop:weil-n}
The following hold.
\begin{enumerate}
	\item Let $n \in \ZZ$.  Let
\[
e_n: E[n] \times E[n] \rightarrow \mu_n
\]
denote the usual Weil pairing as in Section~\ref{sec:backweil}.
    Let 
    $D_P, D_Q  \in \Pic_R^0(E)[n]$.
    Suppose
    \[
D_P \sim 
\sum \tau_i \cdot \left( (P_i) - (\mathcal{O})\right), \quad
D_Q \sim 
\sum \tau_i \cdot \left( (Q_i) - (\mathcal{O}) \right).
\]
    Then 
    \[
    W_n (D_P,D_Q) = 
    \prod_{i,j=0}^{r-1} e_n(P_i, Q_j) ^{\overline{\tau_j} \tau_i}.
    \]
    \item
    Finally, when both of the following quantities are defined, and when $R$ is an imaginary quadratic order, with $\alpha \in R$, then
    \[
       {W}_{N(\alpha)}(D_P,D_Q) = {W}_\alpha(D_P,D_Q)^\alpha.
    \]
    \end{enumerate}
\end{proposition}

\begin{proof}
By a linear equivalence, assume that
\[
D_P = \sum_i \tau_i \cdot \left( (P_i) - (\mathcal{O})\right), \quad
D_Q = \sum_j \tau_j \cdot \left( (Q_j+S) - (S) \right).
\]
where $S$ is chosen to avoid intersections of supports.
We have from \eqref{eqn:hpgp}, we have $f_P = \prod_i f_{i,P}^{ \tau_i}$, $f_Q = \prod_i f_{j,Q}^{\tau j}$ where
\[
	\div(f_{i,P}) = n(P_i) - n(\mathcal{O}), \quad
	\div(f_{j,Q}) = n(Q_j+S) - n(S).
    \]
    We obtain\footnote{In counterpoint to the footnote in the proof of Theorem~\ref{thm:tateprops}, we do have $\overline{g^\alpha} = g^{\overline{\alpha}}$ when $g \in \GG_m^{\otimes 1}$.}
\begin{align*}
	W_n(D_P,D_Q) &= f_P\left( \sum_j \tau_j ( (Q_j+S) - (S) ) \right) 
	\overline{ f_Q\left( \sum_i \tau_i ( (P_i) - (\mathcal{O}) ) \right) }^{-1} \\
		     &= \prod_j f_P( (Q_j+S) - (S) )^{\overline{\tau_j}} 
		     \overline{ \prod_i f_Q( (P_i) - (\mathcal{O}) )^{\overline{\tau_i}} }^{-1} \\
			 &=
	    \prod_{j} \left(\prod_i f_{i,P}( (Q_j + S) - (S) )^{ \tau_i} \right)^{\overline{\tau_j}}
	   \overline{ \prod_{i} \left(\prod_j f_{j,Q}( (P_i) - (\mathcal{O}) )^{ \tau_j} \right)^{\overline{\tau_i}} 
	   }^{-1} \\ &=
	   \prod_{j} \prod_i f_{i,P}( (Q_j + S) - (S) )^{\overline{\tau_j} \tau_i}
	   \overline{
	   f_{j,Q}( (P_i) - (\mathcal{O}) )^{\overline{\tau_i} \tau_j} 
	   }^{-1} \\
	   &=
	   \prod_{j} \prod_i f_{i,P}( (Q_j + S) - (S) )^{\overline{\tau_j} \tau_i}
	   \left(f_{j,Q}( (P_i) - (\mathcal{O}) )^{\overline{\tau_j} {\tau_i}}\right)^{-1}
   \end{align*}
    That shows the first statement.  For the second, 
    suppose $\div(f_P) = \overline\alpha \cdot D_P$ and $\div(f_Q) = \alpha \cdot D_Q$.  
    Observe that for any divisor $D_Q$ with sufficiently disjoint support,
\[
	\frac{(f_{P}^{{\alpha}})(D_Q)}{\overline{(f_Q^{\overline{\alpha}})(D_P)}}
	= 
	\left(\frac{f_{P}(D_Q)}{\overline{f_Q(D_P)}}\right)^{{\alpha}}.
\]
On the left, this is a representative of $W_n(D_P,D_Q)$ in $\GG_m^{\otimes_\ZZ R}[n]$, since $\div(f_P^\alpha) = \alpha \cdot \div(f_P) = n D_P$ and $\div(f_Q^{\overline{\alpha}}) = \overline\alpha \cdot \div(f_Q) = n D_Q$.  However, looking at the right, this is also a representative of $W_\alpha(D_P,D_Q)^\alpha$ in $\GG_m^{\otimes_\ZZ R}[\overline\alpha]$.
\end{proof}

\begin{remark}
	Because of the footnote in the proof of Theorem~\ref{thm:tateprops}, the last displayed equation of the proof above does not necessarily hold when $R$ is a quaternion algebra.  Furthermore, if one is interested in the second statement of the theorem, in the case of $R$ a quaternion algebra, one could use the definition in Theorem~\ref{thm:alternate-weil} as the primary definition of the Weil pairing, but then one may wish to reprove Theorem~\ref{thm:weilprops}; we have not attempted this.
\end{remark}

When $E$ has CM by $\alpha \in R$, and $R$ is an imaginary quadratic order, then there is an alternate definition along the lines of the second definition in Section~\ref{sec:backweil}.  
	Observe that for any field $K$ containing the $n$-th roots of unity, where $n = N(\alpha)$, we have $(K^*)^{\otimes_\ZZ R}[\overline \alpha] \cong (R/nR)[\overline \alpha] \cong R/R\overline \alpha R$.  Observe that this abstract group isomorphism actually constitutes a type of discrete logarithm:  that is, choosing a generator $x$ of the $n$-th roots of unity in $K$, $x^\beta \mapsto \beta$.

\begin{theorem}
\label{thm:weilnon}
	Let $\alpha \in R$ have norm $n = N(\alpha)$.
	Let $\overline{K}$ be an algebraically closed field with characteristic coprime to $n$.  Suppose $n$ is also coprime to the discriminant of $R$.
    The pairing
    \[
	    {W}_\alpha : \Pic_R^0(E)[\overline{\alpha}](\overline{K}) \times \Pic_R^0(E)[\alpha](\overline{K}) \rightarrow (R/nR) [\overline\alpha]
\]
is non-degenerate.
\end{theorem}

\begin{proof}
	As in the proof of Theorem~\ref{thm:tatenon}, for $W_n$ it suffices to use Proposition~\ref{prop:weil-n} and the non-degeneracy of $e_n$ (Proposition~\ref{prop: weilprop}).
Now consider the general case.  Fix $D_P \in \Pic_R^0(E)[\overline{\alpha}](\overline{K})$.   Suppose $W_\alpha(D_P,D_Q) = 1$ for all $D_Q \in \Pic_R^0(E)[\alpha](\overline{K})$.  Then for all $D_Q \in \Pic_R^0(E)[N(\alpha)](\overline{K})$, we have $\overline{\alpha} \cdot D_Q \in \Pic_R^0(E)[\alpha](\overline{K})$, and therefore $W_{N(\alpha)}(D_P,D_Q) = W_{\alpha}(D_P, \overline{\alpha} \cdot D_Q) = 1$.  So we have $D_P \sim 0$ by the first case.
\end{proof}

\section{Curves with complex multiplication}
\label{sec:CM}

Thus far the pairings we have constructed are somewhat abstract, being defined even for elliptic curves having no complex multiplication.  In this section, we pull back these pairings to curves with complex multiplication by subrings of $R$, and see that the resulting pairings are sesquilinear with respect to the endomorphisms.

To be precise, the pairings $T_\alpha$ and $W_\alpha$ are defined on subgroups or quotients of $\Pic_R^0(E)$.  If we have an $R$-module homomorphism into $\Pic_R^0(E)$, then we can pull back the pairing along this $R$-module homomorphism.  In what follows, we define an $R$-module homomorphism $\eta : E \rightarrow \Pic_R^0(E)$ to pull back along, where $E$ is an elliptic curve with complex multiplication, considered as an $R$-module with respect to this CM.

The rest of the section is devoted to the basic properties of these new pairings $\widehat{T}_\alpha$ and $\widehat{W}_\alpha$, analogously to what has been proven for $T_\alpha$ and $W_\alpha$.  We almost immediately restrict to the case of quadratic $R$ for reasons of sanity.  We forewarn the reader that $\eta$ twists the action of $R$, that is, $\eta([\alpha]P) = \overline{\alpha}\eta(P)$, and that this results in all the vincula hopping about like so many excited circus fleas.  In particular, where in the last section our pairings were conjugate linear on the right, in this section they become conjugate linear on the left.

\subsection{Pull-back to CM curves}

Suppose $S \subseteq R$ is a subring, and suppose that $E$ has CM by $S$.  Fix a map $[\cdot]:S \rightarrow \End(E)$, $\gamma \mapsto [\gamma]$.

Then for $\gamma \in S$, $[\gamma]_*$ acts on $\Pic^0(E)$.  Then there is a surjective $R$-module homomorphism
\[
\epsilon: \Pic_R^0(E) \cong R \otimes_{\ZZ} \Pic^0(E) \rightarrow R \otimes_S \Pic^0(E).
\]
which in particular takes
\[
	\gamma \otimes_\ZZ D \rightarrow \gamma \otimes_S D \sim [\gamma]_* D
\]
for all $\gamma \in S$.  This gives rise to an exact sequence of $R$-modules defining $\Pic_{R,S}^0(E)$ as follows:
\begin{equation}
\label{eqn:PicAsExtRS}
\xymatrix{
	0 \ar[r] &\Pic_{R,S}^0(E) \ar[r]^-{\eta} &\Pic^0_{R}(E) \ar[r]^-\epsilon &R \otimes_S\Pic^0(E) \ar[r] &0
}
\end{equation}

With \eqref{eqn:PicAsExtRS}, we can pull back pairings to $\Pic_{R,S}^0(E)$.  When $R=S$, we can identify $\Pic_{R,S}^0(E)$ with $E^{r-1}$ via
\[
	E^{r-1} \rightarrow \Pic_{R,S}^0(E),
\quad
(P_1,\ldots,P_{r-1}) \mapsto
\left(\sum [-\tau_i]P_i\right) - (\mathcal{O}) + \sum \tau_i \left( (P_i) - (\mathcal{O} )\right).
\]
(This is not canonical; there's a choice of automorphism of $E^{r-1}$.)  
Thus we obtain pairings on $E^{r-1}$.  We will focus on the imaginary quadratic case, where we will make this explicit.

\subsection{Quadratic case}

Suppose $E$ defined over $K$ has CM by $R$, an order in an imaginary quadratic field.  To fix a map $R \rightarrow \End(E)$, denoted $\gamma \rightarrow [\gamma]$, we first fix an injection $\iota: R \rightarrow \overline K$, and then we can take that which is normalized as in \cite[II.1.1]{Sil2}, i.e. $[\gamma]^* \omega = \iota(\gamma) \omega$ for the invariant differential $\omega$ of $E$ and $\gamma \in R$.
The situation of the last subsection becomes
\begin{equation}
\label{eqn:PicAsExt}
\xymatrix{
0 \ar[r] &E \ar[r]^-{\eta} &\Pic^0_{R}(E) \ar[r]^-\epsilon &E \ar[r] &0
}
\end{equation}
given by $R$-module homomorphism
\[
\epsilon: \Pic^0_R(E) \rightarrow E, \quad
\sum \alpha_i (P_i) \mapsto \sum [\alpha_i] P_i.
\]
The kernel is an $R$-module, identified with $E$ via
\begin{equation}
  \label{eqn:eta}
\eta: E \rightarrow \Pic_R^0(E),
\quad
P \mapsto
([-\tau]P) - (\mathcal{O}) + \tau ( (P) - (\mathcal{O} )).
\end{equation}
but note that the $R$-module action on this $E$ is twisted:
\begin{equation}
  \label{eqn:flea}
    \eta([\alpha]P) = \overline{\alpha} \cdot \eta(P),
\end{equation}
because
if $\alpha = a + c\tau$ and $\alpha \tau = b + d\tau$, then $\overline\alpha = d - c\tau$ and $\overline\alpha \tau = -b + a\tau$, so
\[
\eta([\alpha]P) = ([-\tau \alpha]P) - (\mathcal{O}) + \tau ( ([\alpha]P) - (\mathcal{O}))
\sim (
d([-\tau]P) - b(P) + \tau ( -c([-\tau]P) + a(P) )
)
=
\overline{\alpha} \cdot \eta(P).
\]
Observe that $\eta$ is not actually dependent on the choice of $\tau$; a map fitting the exact sequence is unique up to automorphism of $E$. 
Notice $\eta$ respects the action of any isogeny $\phi: E \rightarrow E'$ which itself respects CM by $R$, i.e., if $\phi \circ [\tau] = [\tau] \circ \phi$, then
\[
\eta(\phi P) = \phi_* \eta(P).
\]

Finally, we discuss the Galois action.  Let $\sigma \in \operatorname{Gal}(\overline K / K)$.  Recall that 
the exact sequence \eqref{eqn:PicAsExt} depends upon the normalized choice of map $R \rightarrow \End(E)$ and the injection $\iota$.  Write $\eta_E$ and $\eta_{E^\sigma}$ to distinguish.  When we conjugate $E$ to $E^\sigma$, making these normalized choices, there is an isomorphism $\End(E) \cong \End(E^\sigma)$ given by $([\alpha]_{E})^\sigma = [\alpha^\sigma]_{E^\sigma}$ (this follows as in \cite[II.2.2(a)]{Sil2}).  
Then the following commutes:
\begin{equation}
\label{eqn:PicAsExt2}
\xymatrix{
	0 \ar[r] &E \ar[r]^-{\eta_E} \ar[d]^\sigma &\Pic^0_{R}(E) \ar[r] \ar[d]^\sigma &E \ar[r] \ar[d]^\sigma &0 \\
	0 \ar[r] &E^\sigma \ar[r]^-{\eta_{E^\sigma}} &\Pic^0_{R^\sigma}(E^\sigma) \ar[r] &E^\sigma \ar[r] &0
}
\end{equation}
where the notation $R^\sigma$ indicates that we use the injection $\iota \circ \sigma : R \rightarrow K$ in defining $\eta_{E^\sigma}$, i.e. we initially replace $R$ with $R^\sigma$ so that
\[
	\eta_{E^\sigma}: E^\sigma \rightarrow \Pic_{R^\sigma}^0(E),
\quad
P \mapsto
([-\tau^\sigma]P) - (\mathcal{O}) + \tau^\sigma ( (P) - (\mathcal{O} )).
\]

This preserves the Galois action on $\Pic_R^0$ as given before:
      \[
	       (\gamma \cdot \eta_E(P) )^\sigma 
	      = \eta_E([\overline\gamma]_E P)^\sigma 
	      = \eta_{E^\sigma}( [\overline\gamma^\sigma]_{E^\sigma} P^\sigma )
	      = {\gamma} \cdot \eta_{E^\sigma}(P^\sigma)
	      = \gamma \cdot (\eta_E(P))^\sigma.
      \]

\subsection{Pairings for quadratic $R$}

Define  
\[
\widehat{W}_\alpha : E[\overline{\alpha}] \times E[\alpha] \rightarrow \GG_m^{\otimes_\ZZ R} [\alpha],
\quad
\widehat{W}_\alpha(P,Q) := W_{\overline\alpha}( 
{\eta}(P),
\eta(Q)
),
\]
where $\eta$ is as in the previous section.  Observe that $\eta$'s twisting of the $R$-action, \eqref{eqn:flea}, results in many swaps of vincula, when comparing to the domain and codomain of $W_\alpha$.

\begin{theorem}
\label{thm:weilprops-red}
    The pairing defined above is well-defined, bilinear, and satisfies
    \begin{enumerate}
	\item Restricted Sesquilinearity:  For $\gamma, \delta$ such that $\gamma^{(\alpha)} = \gamma$ and $\delta^{(\overline\alpha)} = \delta$, we have 
\[
	\widehat{W}_\alpha([\gamma] P, [\delta] Q)
    = \widehat{W}_\alpha(P,Q)^{\delta\overline\gamma}.
    \]
    \item Conjugate skew-Hermitianity:
    \[
	    \widehat{W}_\alpha(P,Q) = \overline{\widehat{W}_{\overline{\alpha}}(Q,P)}^{-1}.
\]
        \item Compatibility:   
		Let $\phi: E \rightarrow E'$ be an isogeny between curves with CM by $R$ and satisfy $[\alpha] \circ \phi = \phi \circ [\alpha]$.   Then for $P \in E[\overline\alpha]$ and $Q \in E[\alpha]$,
    \[
    \widehat{W}_\alpha(\phi P,\phi Q) = \widehat{W}_\alpha(P,Q)^{\deg \phi}.
    \]
\item Coherence: For $P \in E[\overline{\alpha\beta}]$, $Q \in E[{\alpha\beta}]$,
    \[
	    \widehat{W}_{\alpha\beta}(P,Q) = \widehat{W}_{\alpha}( [\overline{\beta}] P, Q) \in \GG_m^{\otimes_\ZZ R}[\alpha], \quad
	    \widehat{W}_{\alpha\beta}(P,Q) = \widehat{W}_\beta( P, [{\alpha}] Q) \in \GG_m^{\otimes_\ZZ R}[\beta].
    \]
    \item Galois invariance: Suppose $E$ is defined over a field $K$, and suppose there is an injection $\iota: R \rightarrow \overline K$; indicate this in the notation for the pairing as discussed above.
 For $\sigma \in \operatorname{Gal}(\overline K/K)$, 
\[
	\widehat{W}^\iota_\alpha(P,Q)^\sigma = \widehat{W}^{\iota \circ \sigma}_{\alpha}(P^\sigma, Q^\sigma).
\]
    \end{enumerate}

\end{theorem}

In the language of isogeny-based cryptography, the condition on $\phi$ in the compatibility property above is that $\phi$ is \emph{oriented} by $\ZZ[\alpha]$.

\begin{proof}
We see immediately that this pairing is sesquilinear, skew-Hermitian, coherent and compatible, since $\eta$ is a twisted $R$-module homomorphism.  Recalling that $\eta([\alpha]P) = \overline{\alpha} \cdot \eta(P)$, we have to place the vincula carefully.
Galois invariance of $\widehat{W}_\alpha$ follows from Galois invariance of $W_\alpha$, with reference to the discussion at the end of the last section.
\end{proof}

With this language we can obtain an alternate definition of the Weil pairing, analogous to Definition~\ref{defn: weil2} in the classical case.  For the following statement, observe that although $\eta$ is only defined in \eqref{eqn:eta} as taking values in $\Pic_R^0(E)$, we can use the formula of \eqref{eqn:eta} to give a map $\eta$ with the same formula into $\Div_R^0(E)$.  We will use the same notation.  However, it only becomes $R$-linear when considered into $\Pic_R^0(E)$.  

\begin{theorem}
	\label{thm:alternate-weil}
	Suppose $E$ has CM by $R$, an imaginary quadratic order.  Let $\alpha \in R$.
  Let $P \in E[\overline{\alpha}]$ and $D_P \sim \eta(P) \in \Pic^0_R(E)[\alpha]$ such that $[\alpha]^* D_P$ has support disjoint from $\mathcal{O}$.
  Then there is a function $g_P$ with divisor
	$\div(g_P) = [\alpha]^* D_P$.
	Suppose $Q \in E[\alpha]$.  
Then
\[
	\widehat{W}_\alpha(P,Q) = \frac{ g_{P} (\eta(Q + X))}{g_{P}(\eta(X))},
\]
where $X$ is any element of $E$ such that $\eta(X)$ and $\eta(Q+X)$ have support disjoint from $g_P$. 
\end{theorem}

\begin{proof}
\textbf{Formula for $g_P$.}  
Fix $f_P$ to have divisor $\alpha \cdot D_P$ where $D_P \sim \eta(P)$ such that $D_P$ has no support at $\mathcal{O}$.
Since $[\alpha]^*((P)-(\mathcal{O}))$ is principal by the assumption 
that $P \in E[\overline\alpha]$, one sees $[\alpha]^* \eta(P)$ and 
therefore $[\alpha]^* D_P$ are principal.  Therefore let $g_P$ be a function with this divisor.  

We now obtain a formula for $g_P \circ \eta$.  
%
%

Define for any $X \in E$ such that $\eta(X)$ and $\eta(Q+X)$ have support disjoint from $[\alpha]^* D_P$, 
    \[
	    H_{X} := [{\alpha}]_* \eta(X) - \overline{\alpha} \cdot \eta(X) \in \Div_R^0(E).
    \]
    This is principal since $\eta(P) = (\tau + [-\tau]_*)( (P) - (\mathcal{O}) )$ and 
    \[
      (-\tau + [\overline\tau]_*)(\tau + [-\tau]_*) = N(\tau) + [-N(\tau)]_* + \tau \left( [\Tr(\tau)]_* - \Tr(\tau) \right)
    \]
    takes degree-zero divisors to principal divisors.  Write $H_{X} = \div(h_{X})$.  
    Set
    \[
	    g_{P}'(X) :=
	    f_{P}(\eta(X))
    \overline{
	    h_{X}( D_P )
    }.
    \]
    We have 
\begin{align*}
	g_{P}'(X)^{{\alpha}} =&  f_{P}( \overline{\alpha} \cdot \eta(X)){ \overline{h_{X}( {\alpha} \cdot D_P )}  } \\
	=& f_{P}( \overline{\alpha}\cdot \eta(X) +  \div(h_{X}) ) \\
=& f_{P}([{\alpha}]_* \eta(X)  )
\end{align*}
We obtain $g_{P}'^{{\alpha}} = f_{P} \circ  [\alpha] \circ \eta$.  
    Let $\Gamma = [-\tau]^* + \overline{\tau}$.  
   Observe that $\Gamma \div(f) = \div(f \circ \eta)$ for any $f$.
Then
	\[
  {\alpha} \cdot \div(g_P') 
    = 
		\div(g_P'^{{\alpha}}) =
		 \div( ([\alpha]^* f_P) \circ \eta ) 
     = \Gamma [\alpha]^* {\alpha} \cdot D_P 
     = \alpha \cdot \Gamma  [\alpha]^* D_P.
\]
Therefore,
\[
  \div( g_P') = \Gamma [\alpha]^* D_P.
\]
Hence $g_P' = g_P \circ \eta$ up to a choice of scalar multiple.

\textbf{Equivalence of pairing formulas.}
Then, since $Q \in E[\alpha]$, we have $[\alpha]_* \eta(Q+X) = [\alpha]_* \eta(X)$, and so the divisor
\[
	\div(h_{X}) - \div(h_{Q+X}) =  \overline\alpha \cdot (\eta(Q+X)-\eta(X)) - [\alpha]_* (\eta(Q+X) - \eta(X))  = \overline\alpha \cdot (\eta(Q+X)-\eta(X))
\]
is the divisor of a function $f_{Q}$.
We may now compute
\begin{align*}
 \frac{ g_{P}(\eta(Q + X))) }{ g_{P}(\eta(X)) }  
&= \frac{ f_{P}(\eta(Q + X))) \overline{h_{Q+X}(D_P)} }{ f_{P}(\eta(X)) \overline{h_{X}( D_P ) }}\\
&= { f_{P}(\eta(Q+X) - \eta(X)  ) }{ \overline{f_{Q}(D_P)}^{-1} } \\
&= W_\alpha( D_P, \eta(Q+X)-\eta(X)) \\
&= \widehat{W}_\alpha(P,Q).
\end{align*}
\end{proof}

\begin{theorem}
	Let $\alpha \in R$.
	Let $K$ be a finite field with algebraic closure $\overline{K}$ and characteristic coprime to $N(\alpha)$.  Suppose also that $n=N(\alpha)$ is coprime to the discriminant of $R$. 
    The pairing
    \[
	    \widehat{W}_\alpha : E[\overline{\alpha}](\overline{K}) \times E[\alpha](\overline{K}) \rightarrow (R/nR) [\alpha],
\quad
\widehat{W}_\alpha(P,Q) = W_{\overline{\alpha}}( 
{\eta}(P),
\eta(Q)
).
\]
is non-degenerate.
\end{theorem}

Observe that in writing the codomain in the way we do here, we are using a discrete logarithm as described immediately before Theorem~\ref{thm:weilnon}.

\begin{proof}
	Note that $(\overline{K}^*)^{\otimes_\ZZ R}[\alpha] \cong (R/nR)[\alpha]$, as in the proof of Theorem~\ref{thm:weilnon}.
	We will use the alternate definition of $\widehat{W}_\alpha$ in Theorem~\ref{thm:alternate-weil}, and the reader is asked to refer to the notation in that proof.

 	In particular, fix $P \in E[\overline{\alpha}](\overline K)$ and assume that $\widehat{W}_\alpha(P,Q)=1$ for all $Q \in E[\alpha](\overline{K})$.  Then, using the notation of Theorem~\ref{thm:alternate-weil} and its proof, $g_P(\eta(X+Q)) = g_P(\eta(X))$ for all $Q \in E[\alpha](\overline{K})$, where $X \in E(\overline{K})$ need only satisfy appropriate conditions on supports.  So $t_{Q}^*$ fixes $g_P \circ \eta \in (\overline{K}(E)^*)^{\otimes_\ZZ R}$.  

 The map
\[
E[\alpha] \rightarrow \operatorname{Aut}[ \overline K(E) / [\alpha]^*\overline K(E) ], \qquad S \mapsto t_S^*
\]
	is an isomorphism \cite[Thm III.4.10(b)]{Sil1} ($t_S$ denoting translation-by-$S$).  
  Therefore, $g_P \circ \eta = h \circ [{\alpha}]$ for some $h \in (\overline K(E)^*)^{\otimes_\ZZ R}$.  Hence, using $f_P$ as in the proof of Theorem~\ref{thm:alternate-weil},
\[
	h^{\overline\alpha} \circ [\alpha] =
	(h \circ [{\alpha}])^{\overline{\alpha}} = g_P^{\overline{\alpha}} \circ \eta = f_{P} \circ [{\alpha}] \circ \eta = f_P \circ \eta \circ [\alpha],
\]
implying that $f_P \circ \eta = h^{\overline{\alpha}}$.  Taking divisors,
	\[
		\overline\alpha \cdot \div(h) = \div(f_P \circ \eta)
    = \Gamma \div(f_P)
    = \Gamma \overline\alpha \cdot D_P
    = \overline\alpha \cdot \Gamma D_P.
	\]
	From this, we determine that $\Gamma D_P$ is principal.  Recall that $D_P \sim \eta(P) = ([-\tau]P) - (\mathcal{O}) + \tau \left( (P) - (\mathcal{O}) \right)$.  Thus, $\Gamma \eta(P)$ is principal.  Momentarily writing $D' = (P) - (\mathcal{O})$,
	\begin{align*}
		[-\tau]^* \eta(P) + \overline\tau \eta(P) 
		&=
		[-\tau]^*[-\tau]_* D' + N(\tau) D' + \Tr(\tau) [-\tau]_* D' + \tau \left( [-\tau]^*D' - [-\tau]_*D' \right).
	\end{align*}
	From principality, we conclude that, in particular, 
	\[
    [2N(\tau) - \Tr(\tau) \tau]P = [\tau - \overline\tau]P = \sum_{S \in E[-\tau]} S \in E[2].
	\]
  Call this two-torsion point $U$.  We have $U \neq \mathcal{O}$ if and only if the kernel of $-\tau$ is cyclic of even order.  Without loss of generality, we can replace $\tau$ with $\tau + 1$ to avoid this case.  From this,
\[
  [2N(\tau) - \Tr(\tau) \tau]P = [\tau - \overline\tau]P = \mathcal{O}.
	\]
	The norms of these coefficients of $P$ are $-N(\tau) \Delta_R$ and $\Delta_R$, where $\Delta_R$ is the discriminant of $R$.  Recalling that $P \in E[\overline\alpha]$, and that $N(\alpha)$ and $\Delta_R$ are coprime, we can conclude that $P = \mathcal{O}$.  
\end{proof}

We can describe $\widehat{W}_\alpha$ in terms of the usual $\alpha$-Weil pairing.

\begin{theorem}
Let $e_\alpha$ be the $\alpha$-Weil pairing as described in Section~\ref{sec:backweil}.  Then
    \[
     \widehat{W}_\alpha (P,Q) = 
   \left( 
	   e_{\overline\alpha}(P,Q)^{2N(\tau)}
	   e_{\overline\alpha}([-\tau]P,Q)^{\Tr(\tau)} \right) \left(
	   e_{\overline\alpha}([\tau - \overline{\tau}]P,Q)
   \right)^{\tau}.
    \]
        Furthermore, when both of the following quantities are defined,
    \[
       \widehat{W}_{N(\alpha)}(P,Q) = \widehat{W}_\alpha(P,Q)^{\overline\alpha}.
    \]
\end{theorem}

\begin{proof}
We have
\[
  \widehat{W}_\alpha(P,Q)  = W_{\overline\alpha}(\eta(P), \eta(Q)) = g_P( \eta(Q+X) - \eta(X) )
\]
where
\[
  \div(g_P) = [\overline\alpha]^* D_P, \quad D_P \sim \eta(P) = ([-\tau]P) - (\mathcal{O}) + \tau \left( (P) - (\mathcal{O}) \right).
\]
Let us write this as $g_P = g_1 g_2^\tau$, where 
\[
  \div(g_i) = [\overline\alpha]^* D_{P,i}, \quad
  D_{P,1} \sim ([-\tau]P) - (\mathcal{O}), \quad
  D_{P,2} \sim (P) - (\mathcal{O}).
\]
Then we continue,
\[
  \widehat{W}_\alpha(P,Q) 
  = g_1( ([-\tau](Q+X)) - ([-\tau]X) )
  g_1( (Q+X) - (X) )^{\overline\tau}
  g_2( ([-\tau](Q+X)) - ([-\tau]X) )^{\tau}
  g_2( (Q+X) - (X) )^{\tau\overline\tau}.
\]
From this, and Definition~\ref{defn: weil2}, we get
\[
  \widehat{W}_\alpha(P,Q) = 
  e_{\overline\alpha}( [-\tau]P, [-\tau]Q )
  e_{\overline\alpha}( [-\tau]P, Q )^{\overline \tau}
  e_{\overline\alpha}( P, [-\tau]Q )^\tau
  e_{\overline\alpha}( P, Q )^{\tau \overline\tau}.
\]
Applying bilinearity and coherence from Proposition~\ref{prop: weilprop} finishes the first statement.  The second follows immediately from Theorem~\ref{thm:weilprops-red}, Coherence.
\end{proof}

Using the notation of the last subsection, define
\[
\widehat{T}_\alpha : E[\overline\alpha] \times E / [\alpha]E \rightarrow \GG_m^{\otimes_\ZZ R} / (\GG_m^{\otimes_\ZZ R})^{\alpha},
\quad
\widehat{T}_\alpha(P,Q) := T_{\overline{\alpha}}( 
{\eta}(P),
\eta(Q)
).
\]

\begin{theorem}
\label{thm:tateprops-red}
    The pairing defined above is well-defined, bilinear, and satisfies
    \begin{enumerate}
        \item Sesquilinearity:  For $P \in E[\overline\alpha]$ and $Q \in E$,    \[
    \widehat{T}_\alpha([\gamma] P, [\delta] Q)
    = \widehat{T}_\alpha(P,Q)^{{\overline\gamma}{\delta}}.
    \]
        \item Compatibility:   
            Let $\phi: E \rightarrow E'$ be an isogeny between curves with CM by $R$ and satisfy $[\alpha] \circ \phi = \phi \circ [\alpha]$.   Then for $P \in E[\overline\alpha]$ and $Q \in E$,
    \[
    \widehat{T}_\alpha(\phi P,\phi Q) = \widehat{T}_\alpha(P,Q)^{\deg \phi}.
    \]
    \item Coherence:
	    Suppose $P \in E[\overline{\alpha \beta}]$, and $Q \in E/ [\alpha\beta] E$.  Then
    \[
    \widehat{T}_{\alpha \beta}(P, Q) \bmod{ (\GG_m^{\otimes_\ZZ R})^{\alpha}}
    = \widehat{T}_{\alpha}( [\overline\beta]P, Q \bmod [{\alpha}]E ).
    \]
    Suppose $P \in E[\overline{\beta}]$, and $Q \in E/ [\alpha\beta] E$.  Then
    \[
    \widehat{T}_{\alpha \beta}(P, Q) \bmod{ (\GG_m^{\otimes_\ZZ R})^{\beta}}
    = \widehat{T}_{\beta}( P, [{\alpha}]Q \bmod [{\beta}]E ).
    \]
    \item Galois invariance: Suppose $E$ is defined over a field $K$, and suppose there is an injection $\iota: R \rightarrow \overline K$; indicate this in the notation for the pairing as discussed above.
 For $\sigma \in \operatorname{Gal}(\overline K/K)$, 
\[
	\widehat{T}^\iota_\alpha(P,Q)^\sigma = \widehat{T}^{\iota \circ \sigma}_{\alpha}(P^\sigma, Q^\sigma).
\]

    \end{enumerate}

\end{theorem}

\begin{proof}
	The proof is as for Theorem~\ref{thm:weilprops-red}.
\end{proof}

We can describe $\widehat{T}_n$ in terms of the usual $n$-Tate-Lichtenbaum pairing.

\begin{theorem}
\label{thm:cmtatereduc}
Let $t_n$ be the $n$-Tate-Lichtenbaum pairing as described in Section~\ref{sec:backtate}.
    \[
     \widehat{T}_n (P,Q) = 
   \left( 
	   t_n(P,Q)^{2N(\tau)}
	   t_n([-\tau]P,Q)^{\Tr(\tau)} \right) \left(
	   t_n([\tau - {\overline\tau}]P,Q)
   \right)^{\tau}.
    \]
    Furthermore, provided both of the following quantities are defined,
    \[
    \widehat{T}_{N(\alpha)}(P,Q) = \widehat{T}_\alpha(P,Q)^{\overline\alpha} \pmod{(\GG_m^{\otimes_\ZZ R})^{\alpha}}
    \]
\end{theorem}

\begin{proof}
  Using Proposition~\ref{prop:tate-n} and \eqref{eqn:eta}, 
  \[
    \widehat{T}_n(P,Q) = T_n( \eta(P), \eta(Q) )
    = t_n([-\tau]P, [-\tau]Q) t_n([-\tau]P,Q)^{\overline \tau} t_n(P, [-\tau]Q)^{\tau} t_n(P,Q)^{\overline\tau \tau}.
  \]
  Applying bilinearity and coherence from Proposition~\ref{prop:tateprops-classical} gives the first result.  The second follows immediately from Proposition~\ref{prop:tate-n}.
\end{proof}

Our final result is about non-degeneracy.  

\begin{proposition}
	Let $K$ be a finite field, and let $E$ be an elliptic curve defined over $K$.  Let $\alpha \in R$ be such that $N(\alpha)$ is coprime to $char(K)$ and the discriminant of $R$.  Let $N = N(\alpha)$.  Suppose $K$ contains the $N$-th roots of unity, and $E[N] = E[N](K)$.
    Then 
    \[
\widehat{T}_\alpha : E[\overline{\alpha}](K) \times E(K) / [\alpha] E(K) \rightarrow (K^*)^{\otimes_\ZZ R} /  ((K^*)^{\otimes_\ZZ R})^{\alpha},
\] is non-degenerate.  Furthermore, if $P$ has annihilator $\overline\alpha R$, then $T_\alpha(P, \cdot)$ is surjective; and if $Q$ has annihilator $\alpha R$, then  $T_\alpha(\cdot, Q)$ is surjective.
\end{proposition}

\begin{proof}
	First, the target is isomorphic to the finite $R$-module $R/{\alpha}R$, which is a principal ideal ring (using the coprimality to the discriminant).  So we can apply Lemma~\ref{lem:tatenon}, and need only show the non-degeneracy.

Recall that $R = \ZZ[\tau]$ for some $\tau$ and since $N$ is coprime to the discriminant, $N$ is coprime to $\tau - \overline{\tau}$ in the sense that $(N, \tau - \overline\tau) = R$.
First we prove an auxiliary result about $\widehat{T}_N$.
Let $P \in E[N](K)$.  Choose $Q \in E(K)$ so that $t_N([\overline\tau - \tau]P,Q)$ has order $N$ (this must exist since $P$ has order $N$, and $N$ is coprime to $\overline\tau - \tau$).  Then by Theorem~\ref{thm:cmtatereduc},
        \begin{align*}
		\widehat{T}_N(P,Q) &= \left( t_N(P,Q)^{2N(\tau)}t_N([-\tau]P,P)^{\Tr(\tau)} \right) ( t_N([\tau-\overline \tau]P,Q) )^{\tau}.
    \end{align*}
    Thus $\widehat{T}_N$ is non-degenerate on the left.  On the other hand, choosing $Q$ first, then since $\tau - \overline\tau$ is coprime to $N$, there exists $P$ making this non-trivial also.  Hence we have both left and right non-degeneracy.

    Next, we consider general $\alpha$.  Let $P \in E[\overline\alpha](K)$.  Then we can let $\div(f_{\alpha,P}) = {\alpha} \cdot \eta(P)$.  Let $\div(f_{N,P}) = N \cdot \eta(P) =\overline{\alpha}  \alpha  \cdot \eta(P)$.  Then
    \[
    f_{N,P}(\eta(Q)) = f_{\alpha,P}( \eta(Q) )^{\overline\alpha}.
    \]
    This is a representative of $\widehat{T}_N(P,Q)$, and for an appropriate choice of $Q$ modulo $[N]E(K)$, is not an $N$-th power (by the first case above).  Taking this $Q$ modulo $[\alpha]E(K)$, $f_{\alpha,P}(\eta(Q))$, a representative of $\widehat{T}_\alpha(P,Q)$, is not an $\alpha$ power, i.e. non-trivial.

    On the other hand, choose $\beta \in R$ coprime to $\alpha$ with $m := \alpha \beta  \in \ZZ$ and $m$ divides $N$.  Fix non-trivial $Q \in E(K)$ modulo $[\alpha]E(K)$.  We can choose a lift of the form $[\beta] Q'$ modulo $[m]E(K)$ for some $Q' \in E(K)$.  Consider the quantity
    \[
	    f_{m,P}(\eta(Q')), \quad \div(f_{m,P}) = m\eta(P).
    \]
    Then there is some $P \in E[m](K)$ so that the quantity above, as a representative of
    $\widehat{T}_m(P,Q')$, is not an $m$-th power (as $m$ divides $N$, this follows from the first part of the proof).  But the quantity is also a representative of
    $\widehat{T}_\alpha(P,Q) = \widehat{T}_\alpha(P,Q')^\beta$, which is still not an $m$-th power.  
    So $\widehat{T}_\alpha(P,Q')$ is not an $\alpha$ power.  And so $\widehat{T}_\alpha(P,Q)$ is not an $\alpha$ power.
\end{proof}

\subsection{Computation.}

	We end by giving an explicit formula for $\widehat{T}_\alpha(P,Q)$ amenable to computation.  
	This algorithm can be adapted to compute $\widehat{W}_\alpha(P,Q)$ also.  

\begin{algorithm}
	\label{alg:alg}
	Recall Remark~\ref{rem:miller}.
Suppose $a + c\tau = \alpha$, $b + d\tau = \alpha \tau$, $a,b,c,d \in \ZZ$, 
which implies $d - c\tau = \overline{\alpha}$, $-b + a\tau = \overline{\alpha}\tau$.
We take $P \in E[\overline{\alpha}]$, $D_P = \eta(P)$, $\div(f_P) = \alpha \cdot D_P$, $f_P = f_{P,1} f_{P,2}^{\tau}$.
The following divisors are principal:
\[
\div(f_{P,1})  = a([-\tau]P) + b(P)-(a+b)(\mathcal{O}), \quad
\div(f_{P,2})  = c
([-\tau]P) + d(P) - (c+d)(\mathcal{O}).
\]
Choose an auxiliary point $S$ and define $D_Q = D_{Q,1} + \tau \cdot D_{Q,2}$ where
\[
D_{Q,1} = ([-\tau]Q + [-\tau]S) - ([-\tau]S), \quad
D_{Q,2} = (Q + S) - (S).
\]
Note that $D_Q \sim \eta(Q)$.
Then, choosing $S$ so that the necessary supports are disjoint (i.e. the support of $\div(f_{P,i})$ and $D_{Q,j}$ are disjoint for each pair $i$, $j$), the pairing is defined as
\[
\widehat{T}_\alpha(P,Q) := f_P(D_Q) =
	f_{P,1}(D_{Q,1}) f_{P,2}(D_{Q,1})^{\tau}
	\left(
	f_{P,1}(D_{Q,2}) f_{P,2}(D_{Q,2})^{\tau}
\right)^{\overline\tau}
\]
which can also be expressed as
\[
\left(
f_{P,1}(D_{Q,1})f_{P,1}(D_{Q,2})^{\Tr(\tau)}f_{P,2}(D_{Q,2})^{N(\tau)} \right)
\left( f_{P,2}(D_{Q,1})f_{P,1}(D_{Q,2})^{-1}
\right)^{\tau}.
\]
To turn this into an efficient algorithm, observe that we can compute $f_{P,i}(D)$ for any divisor $D$ supported on a constant number of points, in $O(\log \max\{a,b,c,d\})$ steps, as follows.  
Define
\[
	\div(h_{P,n}) = n(P) - ([n]P) - (n-1)(\mathcal{O}).
\]
We can compute $h_{P,n}(D)$ using a double-and-add algorithm \cite{Miller} \cite[\S 26.3.1]{GalbraithBook}, evaluating at $D$ at each step.  Then observe that
\[
	\div(f_{P,1}) = \div(h_{[-\tau]P,a}) + \div(h_{P,b}) + \div(g), \quad \div(g) = ([-a\tau]P) + ([b]P) - 2 (\mathcal{O})
\]
Thus, compute $g(D)$ (the straight line through $[-a\tau]P$ and $[b]P$ in Weierstrass coordinates), and multiply together to compute $f_{P,1}(D) = h_{[-\tau]P,a}(D)h_{P,b}(D)g(D)$.  Computing $f_{P,2}(D)$ is similar.
\end{algorithm}

\section{Examples}

Consider the curve $E: y^2 = x^3 - x$ over the prime field $\FF_q$, $q=401$.  We have $E(\FF_q) = (\ZZ/20\ZZ)^2$.  This curve has complex multiplication by $R = \ZZ[i]$, given by $[i]: (x,y) \mapsto (-x,iy)$, where $i = 20 \in \FF_q$.  Let $\alpha = 1 - 2i$.  Consider the pairing 
\[
\widehat{T}_\alpha : E[\overline{\alpha}](\FF_q) \times E(\FF_q)/[\alpha]E(\FF_q) \rightarrow (\FF_q^*)^{\otimes_\ZZ \ZZ[i]} / ((\FF_q^*)^{\otimes_\ZZ \ZZ[i]})^{\alpha}.
\]
A basis for the $5$-torsion is $P = (204,283) \in E[\overline\alpha](\FF_q)$, $Q = (56,137) \in E[\alpha](\FF_q)$.
Also, $[i]P = (197,46)$, $[i]Q = (345,334)$.
Note that $Q$ generates $E(\FF_q)/[\alpha]E(\FF_q)$ and $P$ generates $E[\overline\alpha](\FF_q)$, each of size $5$.
We will compute $\widehat{T}_\alpha(P,Q)$ in a variety of ways.

\textbf{Method 1.} Let us compute the pairing using Algorithm~\ref{alg:alg}.  We have, for $a=d=1$, $b=2$, $c=-2$, that
\[
a + ci = \alpha, \quad b + di = \alpha \tau, \quad d - ci = \overline{\alpha}, \quad -b + ai = \overline{\alpha}\tau.
\]
Therefore we define
\[
	\div(f_{P,1})  = ([-i]P) + 2(P)-3(\mathcal{O}), \quad
\div(f_{P,2})  = -2
([-i]P) + (P) +(\mathcal{O}).
\]
Recall that $[2]P = [i]P$, since $[\overline\alpha]P = \mathcal{O}$.  Using the notation $L(T,U)$ for the line through $T$ and $U$, having divisor $(T) + (U) - (T+U) - (\mathcal{O})$ and $V(T)$ for the vertical line through $T$, having divisor $(T) + (-T) - 2(\mathcal{O})$, we have from the expression above that
\[
	f_{P,1} 
	= L(P,P).
\]
Therefore, using the standard Weierstrass model and its addition formul\ae,
\[
	f_{P,1}(X,Y) = (Y-\lambda_1 X + \lambda_1 x(P) - y(P))(X - x(2P)), \quad
	\lambda_1 = \frac{3x(P)^2 - 1}{2y(P)}.
\]
This becomes
\[
	f_{P,1}(X,Y) = -47X + Y + 82.
\]
Now for the second function
\[
	\div(f_{P,2})  = -2([-i]P) + (P) + \mathcal{O}
\]
we have
\[
  f_{P,2} = \left(\frac{L([-i]P,[-i]P)}{V([-2i]P)}\right)^{-1} = \frac{V([-2i]P)}{L([-i]P,[-i]P)}.
\]
That is,
\[
  f_{P,2}(X,Y) = \frac{X - x([-2i]P)}{Y-\lambda_2 X + \lambda_2 x([-i]P) - y([-i]P)}, \quad
  \lambda_2 = \frac{3x([-i]P)^2 - 1}{2y([-i]P)}.
\]
This becomes
\[
	f_{P,2}(X,Y) = \frac{X + 197}{-138X + Y - 36}.
\]
Let $h=3$, a multiplicative generator for $\FF_q$.  Note that $\ZZ[i]/\alpha \ZZ[i]$ has representatives $\{0,1,2,3,4\}$, so 
\[
  (\FF_q^*)^{\otimes_\ZZ \ZZ[i]} / ((\FF_q^*)^{\otimes_\ZZ \ZZ[i]})^{\alpha} = \{ 1, h, h^2, h^3, h^4 \}.
\]
Using an auxiliary point such as $S = (0,0)$ and the formula from Algorithm~\ref{alg:alg}, we obtain
\[
  \widehat{T}_\alpha(P,Q) \equiv 175(-5)^{i} \equiv h^{158+248i} \equiv h^{3+3i} \equiv h^{2} \pmod{ h^\alpha }.
\]
Using instead an auxiliary point such as $S=(1,0)$, we obtain
\[
	\widehat{T}_\alpha(P,Q) \equiv 186 \cdot 144^{i} \equiv h^{134+106i} \equiv h^{4+i} \equiv h^{2}.
\]
This illustrates the independence of the choice of $S$.

To take this into $\mu_5^{\otimes_\ZZ \ZZ[i]}$, for the purposes of comparing with the next method, we raise to the $(q-1)/5 = 80$.  Let $g = 72 = h^{80}$, a generator for $\mu_5 = \{ 1, g, g^2, g^3, g^4 \}$.  We obtain a type of \emph{reduced} pairing (albeit slightly different than that of Remark~\ref{rem:red}):
\[
	\widehat{T}^{red}_\alpha(P,Q) := \widehat{T}_\alpha(P,Q)^{\frac{q-1}{5}} \equiv g^{2}.
\]

\textbf{Method 2.}  Now we will compute $\widehat{T}^{red}_\alpha(P,Q)$ by using both parts of Theorem~\ref{thm:cmtatereduc}, relating it to $\widehat{T}_5$.
We have the reduced Tate-Lichtenbaum pairing $t_n^{red} = t_n^{(q-1)/n}$ as implemented in many mathematical software systems,
\[
	t^{red}_5(P,Q) \equiv g, \quad t^{red}_5([2i]P,Q) \equiv g^4, \quad t^{red}_5(P,P) \equiv 1, \quad t^{red}_5([2i]P,P) \equiv 1, \quad t^{red}_5(Q,Q) \equiv 1, \quad t^{red}_5([2i]Q,Q) \equiv 1.
\]
Therefore, by the first part of Theorem~\ref{thm:cmtatereduc},
\begin{equation}
  \label{eqn:prev}
\widehat{T}^{red}_5(P,Q) 
\equiv g^{2-i} \equiv g^4, \quad \widehat{T}^{red}_5(P,P) \equiv g^{0}, \quad \widehat{T}^{red}_5(Q,Q) \equiv g^{0}.
\end{equation}
Since $P$ is an $\alpha$-multiple, we expect $\widehat{T}_5(P, \cdot)$ to be $\overline\alpha$ powers.  Note that $\overline\alpha^{-1} \equiv 3 \pmod{\alpha}$.  Therefore, modulo $\alpha$, we have
\[
	\widehat{T}^{red}_\alpha(P,Q) \equiv (g^{2-i})^3 \equiv g^{1+2i} \equiv g^{2}.
\]
This agrees with Method 1.

Finally, for good measure, we repeat the first part of the computation above, namely $\widehat{T}^{red}_5(P,Q)$, using a single generator for the $\ZZ[i]$-module $E[5]$.  Observe that $E[5] = \ZZ[i]S$, where $S = P+Q$.  In particular, $P = (3 + 4i)S$ and $Q = (3 +i)S$.  We have
\[
\widehat{T}^{red}_5(S,S) 
\equiv g^{4}, \quad \widehat{T}^{red}_5(S,P) \equiv g^{2-4i}, \quad \widehat{T}^{red}_5(S,Q) \equiv g^{2-i}.
\]
We can verify that in fact
\[
	\widehat{T}^{red}_5(P,Q) = \widehat{T}^{red}_5([3+4i]S,[3+i]S) = \widehat{T}^{red}_5(S,S)^{(3-4i)(3+i)} = \widehat{T}^{red}_5(S,S)^{8+6i} \equiv (g^{4})^{3+i} \equiv g^{4},
	\]
  agreeing with \eqref{eqn:prev}.

     \bibliographystyle{plain}
     \bibliography{ellnet}

\end{document}